\documentclass[11 pt, oneside]{amsart}
  
  \usepackage[margin=3cm]{geometry}
\usepackage{amsmath,amssymb,amsthm, mathrsfs, mathtools}
\usepackage{amscd,mathtools}
  \usepackage{geometry}
  \usepackage{graphicx,epstopdf}
  \usepackage{color,xcolor}
  \usepackage{enumerate}
  \usepackage{xspace}
  \usepackage{tipa}
  \usepackage{epsfig}
  \usepackage{pinlabel}
  \usepackage[all]{xy}
  
\usepackage{tikz}
\usepackage{caption, subcaption}
\usepackage{tikz-cd}
\usetikzlibrary{calc,decorations.pathreplacing}
\usepackage{tikz}
\usetikzlibrary{calc}
\usetikzlibrary{arrows}
\usetikzlibrary{decorations.pathreplacing}
\usetikzlibrary{intersections}

\usepackage{mathrsfs}
\usetikzlibrary{matrix}
\usetikzlibrary{positioning}
\DeclareMathAlphabet{\pazocal}{OMS}{zplm}{m}{n}
\tikzset{>=stealth}

  \usepackage{hyperref}

\usepackage{float}

  \newcommand{\calB}{\mathcal{B}}
  \newcommand{\calC}{\mathcal{C}}

  \newcommand{\calI}{\mathcal{I}}
  \newcommand{\calJ}{\mathcal{J}}

  \newcommand{\calN}{\mathcal{N}}

  \newcommand{\calU}{\mathcal{U}}
    \newcommand{\calV}{\mathcal{V}}
  \newcommand{\calW}{\mathcal{W}}

  \newcommand{\EE}{\mathbb{E}}

  \newcommand{\NN}{\mathbb{N}}

  \newcommand{\RR}{\mathbb{R}}

  \newcommand{\ZZ}{\mathbb{Z}}

  \newcommand{\bfa}{\textbf{a}}
  \newcommand{\bfb}{\textbf{b}}
  \newcommand{\bfc}{\textbf{c}}

  \newcommand{\gothic}{\mathfrak}

  \newcommand{\GF}{{\gothic F}}
  \newcommand{\Gg}{{\gothic g}}

  \newcommand{\go}{{\gothic o}}

  \newtheorem{theorem}{Theorem}[section]
  \newtheorem{proposition}[theorem]{Proposition}
  \newtheorem{corollary}[theorem]{Corollary}
  \newtheorem{lemma}[theorem]{Lemma}

  \newtheorem{introthm}{Theorem}
  \newtheorem{introcor}[introthm]{Corollary}

  \theoremstyle{definition}
  \newtheorem{definition}[theorem]{Definition}

  \newtheorem*{claim*}{Claim}

  \newtheorem*{question*}{Question}
  \newtheorem*{answer*}{Answer}
  \newtheorem*{application*}{Application}

  \theoremstyle{remark}
  \newtheorem{remark}[theorem]{Remark}
  \newtheorem*{remark*}{Remark}
  
  \newcommand{\thmref}[1]{Theorem~\ref{#1}}
  \newcommand{\corref}[1]{Corollary~\ref{#1}}
  \newcommand{\lemref}[1]{Lemma~\ref{#1}}
  \newcommand{\propref}[1]{Proposition~\ref{#1}}
  
  \newcommand{\remref}[1]{Remark~\ref{#1}}

  \newcommand{\defref}[1]{Definition~\ref{#1}}
  
  \newcommand{\eqnref}[1]{Equation~\eqref{#1}}

  \newcommand{\pka}{\partial_{\kappa}}

  \newcommand{\sA}{{\sf A}}   
     
  \newcommand{\sC}{{\sf C}}   
  \newcommand{\sD}{{\sf D}}

  \newcommand{\sK}{{\sf K}}   
  \newcommand{\sL}{{\sf L}}     
  \newcommand{\sM}{{\sf M}}

  \newcommand{\sQ}{{\sf Q}}   
  \newcommand{\sR}{{\sf R}}

  \renewcommand{\aa}{{\sf a}}   
  \newcommand{\bb}{{\sf b}}   
  \newcommand{\cc}{{\sf c}}   
  \newcommand{\dd}{{\sf d}}

  \newcommand{\kk}{{\sf k}}   
       
  \newcommand{\mm}{{\sf m}}   
  \newcommand{\nn}{{\sf n}}

  \newcommand{\qq}{{\sf q}}   
  \newcommand{\rr}{{\sf r}}

  \newcommand{\uu}{{\sf u}}

  \newcommand{\supp}{{\rm supp}}

\DeclareMathOperator{\diam}{diam}

  \newcommand{\param}{{\mathchoice{\mkern1mu\mbox{\raise2.2pt\hbox{$
  \centerdot$}}
  \mkern1mu}{\mkern1mu\mbox{\raise2.2pt\hbox{$\centerdot$}}\mkern1mu}{
  \mkern1.5mu\centerdot\mkern1.5mu}{\mkern1.5mu\centerdot\mkern1.5mu}}}

\DeclarePairedDelimiterX{\norm}[1]{\lvert}{\rvert}{#1}
\DeclarePairedDelimiterX{\Norm}[1]{\lVert}{\rVert}{#1}

  \newcommand{\st}{\mathbin{\mid}} 
  \newcommand{\ST}{\mathbin{\Big|}} 
  \newcommand{\from}{\colon\thinspace}

  \newcommand{\A}{{A(\Gamma)}}

\newcommand{\CAT}{\ensuremath{\operatorname{CAT}(0)}\xspace}

\newcommand{\myGlobalTransformation}[2]
{
    \pgftransformcm{1}{0}{0.5}{0.3}{\pgfpoint{#1}{#2}}
}

\begin{document}

  \title[Sublinearly Morse Boundary I: \CAT Spaces]
  {Sublinearly Morse Boundary I: \CAT Spaces}
  
%
 \author{Yulan Qing}
 \thanks{Shanghai Center for Mathematical Sciences, Fudan University, Shanghai, China \url{qingyulan@fudan.edu.cn}.}
 
 \author   {Kasra Rafi}
 \thanks{Department of Mathematics, University of Toronto, Toronto, ON, \url{rafi@math.toronto.edu}.}
 
 
  \date{\today}

\begin{abstract} 
To every Gromov hyperbolic space $X$ one can associate a space at infinity
called the Gromov boundary of $X$.  Gromov showed that 
quasi-isometries of hyperbolic metric spaces induce homeomorphisms on
their boundaries, thus giving rise to a well-defined notion of the
boundary of a hyperbolic group. Croke and Kleiner showed that the visual
boundary of non-positively curved (CAT(0)) groups is not well-defined,
since quasi-isometric CAT(0) spaces can have non-homeomorphic boundaries.

We attempt to construct an analogue of the Gromov boundary that encodes the 
hyperbolic directions in a metric space. To this end, for any sublinear function $\kappa$, 
we define a subset of the visual boundary called the $\kappa$--Morse boundary. 
We show that, equipped with a coarse notion of visual topology, this space is QI-invariant 
and metrizable. That is to say, the $\kappa$--Morse boundary of a CAT(0) group is 
well-defined. In the case of Right-angled Artin groups, it is shown in the Appendix that 
the Poisson boundary of random walks is naturally identified with the 
$(\sqrt{t \log t})$--boundary. 
\end{abstract}

\maketitle

\section{Introduction}

To every Gromov hyperbolic space $X$ one can associate a space at infinity 
$\partial X$ called the \emph{Gromov boundary} of $X$. The space $\partial X$ 
consists of equivalence classes of geodesic rays, where two rays are equivalent if they 
stay within bounded distance of each other, and is equipped with the visual topology. 
This boundary is a fundamental tool for studying hyperbolic groups and hyperbolic 
spaces (for example, see  \cite{Kapovichsurvey}). As shown by Gromov  \cite{gromov}, 
quasi-isometries between hyperbolic metric spaces induce homeomorphisms between 
their boundaries, thus giving rise to a well-defined notion of the boundary of a hyperbolic 
group. 

However, this is not true under weaker assumptions. 
In particular, for \CAT spaces, Croke and Kleiner \cite{crokekleiner} showed that visual 
boundaries of \CAT spaces are generally not quasi-isometrically invariant and hence one 
cannot talk about the visual boundary of a \CAT group.  In \cite{qing1} Qing showed that 
even if we restrict our attention to \emph{rank-1} geodesics, 
the space of all rank-1 geodesics is still not quasi-isometry invariant. In \cite{cashen} 
Cashen showed that the subset of the visual boundary consisting of only the Morse
geodesics (equipped with the usual cone topology) is not in general preserved 
by quasi-isometries. 

The correct analogue of the Gromov boundary should consist of all the hyperbolic 
directions in a given metric space. It turns out, many arguments in the study of 
Gromov hyperbolic spaces can still be carried out with sub-linear error terms 
rather than uniform ones. This is our guiding principle as we attempt find correct 
generalizations of the fundamental notions in Gromov hyperbolic spaces and construct 
the new boundary.  In this paper, we introduce a boundary for \CAT spaces that is strictly
larger than the set of Morse geodesics and is equipped with a coarse notion 
of cone topology that makes it invariant under quasi-isometries. 
However, this principle is also applicable beyond the setting of \CAT spaces 
(see \remref{Rem:Developments} for further developments). 

The points in this boundary are geodesic rays that behave like geodesics
in a Gromov hyperbolic space with a sublinear error term. More precisely, 
they satisfy one of the following two equivalent characterizations. 
Given a base-point $\go$ in $X$, define the norm of a point $x$ to be 
$\Norm{x} = d_X(\go, x)$. Now, fixing a sublinear function $\kappa$, 
we say a geodesic ray $b \from [0, \infty) \to X$ starting from $\go$ is $\kappa$--Morse 
if there is a Morse gauge function $\mm_b \from \RR_+^2 \to \RR_+$ such that
if $\zeta$ is a $(\qq, \sQ)$--quasi-geodesic segment with end points on $b$
then, for every point $x$ on $\zeta$, we have 
\[
d_X(x, b) \leq \mm_b(\qq, \sQ) \cdot \kappa(\Norm{x}). 
\]
Alternatively, we say $b$ is $\kappa$--contracting if there exists a constant $\cc_b$ such 
that, for any metric ball $B$ centered at $x$ that is disjoint from $b$, the projection of $B$ 
to $b$ has diameter at most $\cc_b \cdot \kappa(\Norm x)$. 
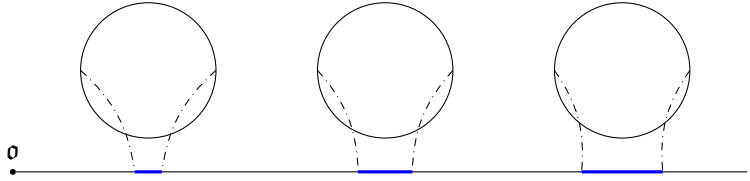
\begin{figure}[h!]
\begin{tikzpicture}[scale=0.9]
 \tikzstyle{vertex} =[circle,draw,fill=black,thick, inner sep=0pt,minimum size=.5 mm]
 
[thick, 
    scale=1,
    vertex/.style={circle,draw,fill=black,thick,
                   inner sep=0pt,minimum size= .5 mm},
                  
      trans/.style={thick,->, shorten >=6pt,shorten <=6pt,>=stealth},
   ]
   
    \node[vertex] (a) at (-1,0) [label=$\go$]  {};
    \node(b) at (10, 0) {};
    \draw(a)--(b){};
    \draw (1,1.5) circle (1);
     \draw (4.5,1.5) circle (1);
      \draw (8,1.5) circle (1);
      
      \draw [-, dash dot] (0, 1.5) to [bend left=20] (0.5+0.3, 0);
       \draw[-, dash dot] (2, 1.5) to [bend right=20](1.5-0.3, 0);
       
        \draw[-, blue, very thick](0.8, 0) to (1.2, 0);
       
        \draw [-, dash dot](3.5, 1.5)to [bend left=20](3.5+0.6, 0);
        \draw [-, dash dot](5.5,1.5)to [bend right=20](5.5-0.6,0);
        
        \draw[-, blue, very thick](4.1, 0) to (4.9, 0);
        
        \draw[-, dash dot](7, 1.5)to [bend left=20](6.5+0.9, 0);
        \draw[-, dash dot](9, 1.5)to [bend right=20](9.5-0.9, 0);
        
        \draw[-, blue, very thick](7.4, 0) to (8.6, 0);
   
\end{tikzpicture}
\caption{Along a $\kappa$--contracting geodesic ray, the diameter of the projection 
of a disjoint ball is allowed to grow at a rate comparable to $\kappa$.}
\label{Fig:Contracting}
\end{figure}

Recall that geodesic rays in Gromov hyperbolic spaces are $\kappa$--contracting
and $\kappa$--Morse for $\kappa =1$. 

\begin{introthm}
A geodesic ray is $\kappa$--Morse if and only if it is $\kappa$--contracting. 
\end{introthm} 

We define the \emph{$\kappa$--Morse boundary} of $X$, which we denote by $\pka X$, 
to be the space of all such geodesic rays and we equip this space
with a notion of visual topology on quasi-geodesics (see Section \ref{Sec:Topology}). 
In the case where $X$ is a Gromov hyperbolic space, $\partial_\kappa X$ is the 
same as the Gromov boundary of $X$ for every function $\kappa$. 

\begin{introthm}
If $\Phi \from X \to Y$ is a quasi-isometry between proper \CAT metric spaces
$X$ and $Y$, then $\Phi$ induces a homeomorphism 
$\Phi^\star \from \partial_\kappa X \to \partial_\kappa Y$.
\end{introthm} 

Therefore one can define the $\kappa$--Morse boundary for any 
group that acts geometrically on a \CAT space or generally any space
that is quasi-isometric to a \CAT space. 

\begin{introcor} 
If $G$ acts quasi-isometrically, discretely and co-compactly on two \CAT spaces
$X_1$ and $X_2$, then for any $\kappa$, the space 
$\partial_\kappa X_1$ is homeomorphic to $\partial_\kappa X_2$. 
Hence, the $\kappa$--Morse boundary $\partial_\kappa G$ of $G$ is well defined. 
\end{introcor} 

Our choice of topology seems to be a natural one, especially since 
$\pka X$ has good topological properties. 

\begin{introthm} \label{Intro:Metrizable}
For every proper \CAT space $X$, $\partial_\kappa X$ is metrizable.  
\end{introthm} 

We also show that the $\kappa$--boundaries associated to different 
sublinear functions are topological subspaces of each other. 

\begin{introthm} 
If $X$ is a \CAT metric space and $\kappa \leq \kappa'$ are two sublinear functions 
then
\[
\partial_\kappa X \subseteq \partial_{\kappa'} X 
\]
where the topology of $\partial_\kappa X$ is the subspace topology 
associated to the inclusion. 
\end{introthm} 

A motivation for this definition of the boundary is the study of random walks on \CAT groups. 
Given a group $G$ and a probability measure $\mu$ on it, the \emph{Poisson boundary} of 
$(G, \mu)$ is a canonical measurable $G$-space which classifies all possible asymptotic 
behaviours of a random walk on $G$ driven by $\mu$ (see the Appendix for precise definitions). 

The boundary depends on the choice of measure, and it is an important open problem 
\cite[page 153]{kaima} whether two finitely supported 
generating measures on the same group give rise to isomorphic boundaries. 
This question is the probabilistic analog of the quasi-isometry invariance question: indeed, 
two generating sets for $G$ give rise to both two quasi-isometric metrics on $G$ and to two 
finitely supported measures. 

In the case $G$ is a right-angled Artin group, in the Appendix (by Y. Qing and G. Tiozzo) we prove the following:

\begin{introthm}\label{T:intro-Poisson}
The $\sqrt{t \log t}$--boundary of $A(\Gamma)$  is a QI-invariant 
topological model for the Poisson boundary of $A(\Gamma)$ associated to 
any random walk with finite support. 
\end{introthm}

To our knowledge, the $\kappa$--Morse boundary defined in this paper is the first 
boundary that is both invariant under quasi-isometries and a model for the Poisson 
boundary. By comparison, the visual boundary is known to be a model of the Poisson 
boundary for \CAT groups but it is not QI-invariant, while the Morse boundary
(\cite{contracting}) is quasi-isometrically invariant but has zero measure with respect 
to random walks, hence, in general, it is not a model for the Poisson boundary. 

The function $\sqrt{t \log t}$ arises from the fact that a generic trajectory of the random 
walk spends a logarithmic amount of time in each flat (\cite{relativehyperbolic}, see also 
Theorem \ref{T:log-exc}).  As shown in \cite{qt}, the same logarithmic excursion property 
also holds for generic elements with respect to the uniform measure on balls in the Cayley 
graph of $G$. This suggests that the $\kappa$--Morse boundary should have full measure 
not only with respect to the hitting measure for random walks but also with respect to the 
Patterson-Sullivan type measure obtained as a weak limit of uniform measures on balls
(in fact, since the first draft of this paper, this has been proven in \cite{GQR}).

\subsection*{History}
Our work builds on previous attempts to construct a boundary for a \CAT group
that is quasi-isometry invariant. Charney and Sultan \cite{contracting} defined a 
contracting geodesic ray in $X$ to be one such that all disjoint balls project to sets 
of diameter at most $\sD$ for some $\sD\geq 0$. They call the set of all such geodesic
rays the \emph{contracting boundary} or the \emph{Morse boundary of $X$}. They equip 
this space with a \emph{direct limit topology} and show that it is invariant under 
quasi-isometries. But this space does not have good topological properties, for example, 
it is not first countable. Cashen-Mackay \cite{cashenmackay}, 
following the work of Arzhantseva-Cashen-Gruber-Hume \cite{Hume}, 
defined a different topology on the Morse boundary of $X$.
They showed that it is Hausdorff and when there is a geometric action
by a countable group, it is also metrizable. In fact, their definition works
for every geodesic metric space. 

The approach in \cite{Hume, cashenmackay} uses a different notion of sublinearly contracting 
geodesic. In  \cite{Hume, cashenmackay}, the contraction is sublinear with respect to the radii of the disjoint 
balls. This is a natural extension of the notion of a Morse geodesic to the setting of 
general metric spaces. But this boundary is smaller than the one defined
in this paper and, in particular, cannot be used as a model for the Poisson boundary. 

It is likely that, when $\kappa=1$, $\pka X$ is the same topological space 
as the Morse boundary equipped with the topology defined in \cite{cashenmackay}. 
If so, \thmref{Intro:Metrizable} would imply that the Morse boundary of every proper \CAT
space is metrizable. 

\begin{remark} \label{Rem:Developments}
Since the first draft of this paper, there has been many developments in advancing the theory
of sublinearly Morse boundaries. For instance, the first named author and Zalloum \cite{QZ21} proved 
that a homeomorphism on the sublinearly Morse boundary comes from a quasi-isometry if and only if the map 
is quasi-m\"obius and sequentially stable. Zalloum  \cite{Zal20} proved that sublinearly Morse boundaries of 
proper \CAT spaces are visibility spaces; Furthermore, Murray, Qing and Zalloum \cite{MQZ20} also showed 
that the $\kappa$--lower divergence of a sublinearly Morse geodesic ray is superlinear. These results provide
further evidence of the similarity between sublinearly Morse boundaries and Gromov boundaries. 
In an upcoming paper \cite{GQR} the claim of Theorem~\ref{T:intro-Poisson} will be extended to all finitely 
generated \CAT groups. We also show that the generic point in the visual boundary of a \CAT space 
with respect to any random walk measure or the Patterson-Sullivan measure is sublinearly Morse. 
This shows that the sublinearly Morse directions are generic with respect to many different notions of generic. 

In a sequel to this paper \cite{QRT21}, we construct sublinearly Morse boundaries for all proper geodesic 
spaces. However, there are substantial differences between the two constructions. In the \CAT setting, many 
of the arguments are simpler and some of the key results have different statements. For example, 
\thmref{Thm:TFAE} does not hold in general. Also, we present further applications in \cite{QRT21}, proving 
statements  analogous to Theorem~\ref{T:intro-Poisson} for mapping class groups and relatively hyperbolic 
groups. 
\end{remark}

\subsection*{Outline of the paper}
Section 2 contains some needed properties of CAT(0) geometry.
In Section 3, we give several equivalent definitions for the notion of 
$\kappa$--contracting geodesic. In Section 4, we define a topology 
for $\pka X$ and establish some topological properties, including the metrizability. 
In Section 5, we define the boundary of a \CAT group, in particular we show 
that $\pka X$ is invariant under quasi-isometry. 
In the last section we examine the group $A=\ZZ \star \ZZ^{2}$ to illustrate in full detail 
the properties of sublinearly Morse boundaries for this example. In particular, we show that the 
$\log$--boundary is a metric model for the Poisson boundary of $A$. 
The Poisson boundary of right-angled Artin groups in general is 
treated in the Appendix. 

\subsection*{Acknowledgments} We thank Jason Behrstock, Ruth Charney, Matthew Cordes,
Talia Fernos, Joseph Maher, Sam Taylor, Abdul Zalloum and especially Chris Cashen and 
Giulio Tiozzo for helpful conversations and comments on earlier versions of this paper.

\section{Preliminaries}

\subsection*{Quasi-Isometry and Quasi-Isometric Embeddings}

\begin{definition}[Quasi Isometric embedding] \label{Def:Quasi-Isometry} 
Let $(X , d_X)$ and $(Y , d_Y)$ be metric spaces. For constants $\kk \geq 1$ and
$\sK \geq 0$, we say a map $\Phi \from X \to Y$ is a 
$(\kk, \sK)$--\textit{quasi-isometric embedding} if, for all $x_1, x_2 \in X$
$$
\frac{1}{\kk} d_X (x_1, x_2) - \sK  \leq d_Y \big(\Phi (x_1), \Phi (x_2)\big) 
   \leq \kk \, d_X (x_1, x_2) + \sK.
$$
If, in addition, every point in $Y$ lies in the $\sK$--neighbourhood of the image of 
$\Phi$, then $f$ is called a $(\kk, \sK)$--quasi-isometry. When such a map exists, $X$ 
and $Y$ are said to be \textit{quasi-isometric}. 

A quasi-isometric embedding $\Phi^{-1} \from Y \to X$ is called a \emph{quasi-inverse} of 
$\Phi$ if for every $x \in X$, $d_X(x, \Phi^{-1}\Phi(x))$ is uniformly bounded above. 
In fact, after replacing $\kk$ and $\sK$ with larger constants, we assume that 
$\Phi^{-1}$ is also a $(\kk, \sK)$--quasi-isometric embedding, 
\[
\forall x \in X \quad d_X\big(x, \Phi^{-1}\Phi(x)\big) \leq \sK \qquad\text{and}\qquad
\forall y \in Y \quad d_Y\big(y, \Phi\,\Phi^{-1}(x)\big) \leq \sK.
\]
\end{definition}

\begin{definition}[Quasi-Geodesics] \label{Def:Quadi-Geodesic} 
A \emph{geodesic ray} in $X$ is an isometric embedding $b \from [0, \infty) \to X$. We fix 
a base-point $\go \in X$ and always assume 
that $b(0) = \go$, that is, a geodesic ray is always assumed to start from 
this fixed base-point. A \emph{quasi-geodesic ray} is a continuous quasi-isometric 
embedding $\beta \from [0, \infty) \to X$ again starting from $\go$. 
The additional assumption that quasi-geodesics are continuous is not necessary,
but it is added for convenience and to make the exposition simpler. 

If $\beta \from [0,\infty) \to X$ is a $(\qq, \sQ)$--quasi-isometric embedding, 
and $\Phi \from X \to Y$ is a $(\kk, \sK)$--quasi-isometry then the composition 
$\Phi \circ \beta \from [t_{1}, t_{2}] \to Y$ is a quasi-isometric embedding, but it may 
not be continuous. However, one can adjust the map slightly to make it continuous 
(see  \cite[Lemma III.1.11]{CAT(0)reference}). Abusing notation, we denote the new map 
again by $\Phi \circ \beta$. Following  \cite[Lemma III.1.11]{CAT(0)reference},
we have that $\Phi \circ \beta$ is a $(\kk\qq, 2(\kk\qq + \kk \sQ + \sK))$--quasi-geodesic.  

Similar to above, a \emph{geodesic segment} is an isometric embedding 
$b \from [t_{1}, t_{2}] \to X$ and a \emph{quasi-geodesic segment} is a continuous 
quasi-isometric embedding $\beta \from [t_{1}, t_{2}] \to X$. 
\end{definition}

\subsection*{Basic properties of \CAT spaces}
A proper geodesic metric space $(X, d_X)$ is \CAT if geodesic triangles in $X$ are at 
least as thin as triangles in Euclidean space with the same side lengths. To be precise, for any 
given geodesic triangle $\triangle pqr$, consider the unique triangle 
$\triangle \overline p \overline q \overline r$ in the Euclidean plane with the same side 
lengths. For any pair of points $x, y$ on edges $[p,q]$ and $[p, r]$ of the 
triangle $\triangle pqr$, if we choose points $\overline x$ and $\overline y$  on 
edges $[\overline p, \overline q]$ and $[\overline p, \overline r]$ of 
the triangle $\triangle \overline p \overline q \overline r$ so that 
$d_X(p,x) = d_\EE(\overline p, \overline x)$ and 
$d_X(p,y) = d_\EE(\overline p, \overline y)$ then,
\[ 
d_{X} (x, y) \leq d_{\EE^{2}}(\overline x, \overline y).
\] 

For the remainder of the paper, we assume $X$ is a proper \CAT space. A metric space $X$ 
is {\it proper} if closed metric balls are compact.
Here, we list some properties of proper \CAT spaces that are needed later (see 
 \cite{CAT(0)reference}). 

\begin{lemma} \label{Lem:CAT} 
 A proper \CAT space $X$ has the following properties:
\begin{enumerate}[i.]
\item It is uniquely geodesic, that is, for any two points $x, y$ in $X$, 
there exists exactly one geodesic connecting them. Furthermore, $X$ is contractible 
via geodesic retraction to a base point in the space. 
\item The nearest-point projection from a point $x$ to a geodesic line $b$ 
is a unique point denoted $x_b$. In fact, the closest-point projection map
\[
\pi_b \from X \to b
\]
is Lipschitz. 
\end{enumerate}
\end{lemma}

\begin{remark} \label{Rem:Projection}
Let $Z$ be a closed subset of $X$. For $x \in X$, we often denote the set of the nearest 
points in $Z$ to $x$ by $x_{Z}$. We also write $d_X(x, Z)$ to mean the distance
between $x$ and the set $Z$, that is $d_X(x, Z)=d_X(x,y)$ for any $y \in x_{Z}$. 
We often think of a geodesic or a quasi-geodesic as a subset of $X$ instead of a map. 
For example, for $x \in X$ and a quasi-geodesic $\beta$, we write 
$d_X(x, \beta)$ to mean the distance between $x$ and the image of $\beta$ in $X$. 
\end{remark} 

We show that if a geodesic segment is ``perpendicular'' to a 
quasi-geodesic, then the concatenation of the geodesic segment with the quasi-geodesic 
is also quasi-geodesic. Given a quasi-geodesic $\beta$, we use $[\param, \param]_{\beta}$ 
to denote the segment of $\beta$ between two specified points.

\begin{lemma}\label{concatenation}
Consider a point $x \in X$ and a $(\qq, \sQ)$--quasi-geodesic segment $\beta$ 
connecting a point $z \in X$ to a point $w \in X$. Let $y$ be a point in $x_\beta$, 
and let $\gamma$ be the concatenation of the geodesic segment $[x, y]$ and the 
quasi-geodesic segment $[y, z]_\beta \subset \beta$. Then 
$\gamma = [x, y] \cup [y, z]_\beta$ is a $(3\qq, \sQ)$--quasi-geodesic.
\end{lemma}

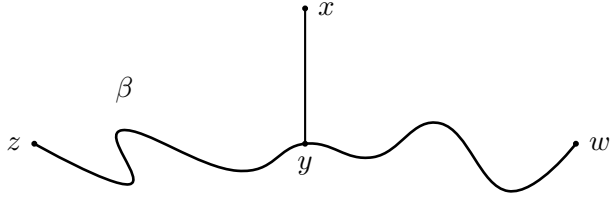
\begin{figure}[h!]
\begin{tikzpicture}[scale=0.9]
 \tikzstyle{vertex} =[circle,draw,fill=black,thick, inner sep=0pt,minimum size=.5 mm]
[thick, 
    scale=1,
    vertex/.style={circle,draw,fill=black,thick,
                   inner sep=0pt,minimum size= .5 mm},
                  
      trans/.style={thick,->, shorten >=6pt,shorten <=6pt,>=stealth},
   ]

  \node[vertex] (z) at (0,0)[label=left:$z$] {}; 
  \node[vertex] (w) at (8,0) [label=right:$w$]  {}; 
  \node[vertex] (x) at (4,2) [label=right:$x$]  {};     
  \node[vertex] (y) at (4,0) [label=below:$y$]  {};     
  \node (a) at (.9,.8) [label=right:$\beta$]  {};    
  \draw[thick]  (4, 2)--(4, 0){};
 
  \pgfsetlinewidth{1pt}
  \pgfsetplottension{.75}
  \pgfplothandlercurveto
  \pgfplotstreamstart
  \pgfplotstreampoint{\pgfpoint{0cm}{0cm}}  
  \pgfplotstreampoint{\pgfpoint{1.4cm}{-.6cm}}   
  \pgfplotstreampoint{\pgfpoint{1.3cm}{.2cm}}
  \pgfplotstreampoint{\pgfpoint{3cm}{-.4cm}}
  \pgfplotstreampoint{\pgfpoint{4cm}{0cm}}
  \pgfplotstreampoint{\pgfpoint{5cm}{-.2cm}}
  \pgfplotstreampoint{\pgfpoint{6cm}{.3cm}}
  \pgfplotstreampoint{\pgfpoint{7cm}{-.7cm}}
  \pgfplotstreampoint{\pgfpoint{8cm}{0cm}}
  \pgfplotstreamend
  \pgfusepath{stroke} 
  \end{tikzpicture}
  
\caption{For $y \in x_{\beta}$, the concatenation of the geodesic segment $[x,y]$ 
and the quasi-geodesic segment $[y,z]_\beta$ is a quasi-geodesic.}
\label{Fig:omega} 
\end{figure}

\begin{proof}
Consider $\gamma \from [t_0, t_2] \to X$ and let $t_1 \in [t_0, t_2]$ be the time
when $\gamma(t_1) = y$, the restriction of $\gamma$ to $[t_0, t_1]$ is the parametrization 
of $[x,y]$ given by arc length and the restriction of $\gamma$ to $[t_1, t_2]$ is the 
parametrization of $[y, z]_\beta$ given by $\beta$. To show that $\gamma$ is a 
quasi-geodesic, we need to estimate the distance between a point in $[x,y]$ and a point in 
$[y,z]_\beta$. However, it is enough to show that $d_X(x,z)$ 
is comparable to $|t_2-t_0|$ because the argument for any other points along $[x,y]$
and along $[y,z]_\beta$ is the same. We argue in two cases. 

\subsection*{Case 1} Suppose $2 d_{X}(x, y) \geq d_{X}(z, y)$. Then,
\[
3d_X(x, y) \geq d_X(z, y)+ d_X(x, y) 
\]
Therefore, 
\begin{align*}
d_X(x, z) &\geq d_X(x,y)  \geq \frac 13 \big(d_X(z, y)+ d_X(x, y) \big)\\
& \geq \frac 13 \left(\frac {1}{\qq} |t_2- t_1| - \sQ + |t_1-t_0|\right) \\
                & \geq \frac{1}{3\qq} |t_2-t_0| -  \frac{\sQ}{3}.\\            
\end{align*}
\subsection*{Case 2} Suppose $2 d_{X}(x, y) < d_{X}(z, y)$, then 
\[
3d_X(x,y) \leq d_X(z, y) + d_X(x,y)
\qquad\Longrightarrow\qquad
2d_X(x,y) \leq \frac 23 \big( d_X(z, y) + d_X(x,y) \big). 
\]
We have
\begin{align*}
d_X(x, z) &\geq d_X(z, y) - d_X(x, y) = d_X(z, y) + d_X(x,y) - 2d_X(x,y)\\
             & \geq \big(d_X(z, y) + d_X(x,y)\big) - \frac 23\big(d_X(z, y)  + d_X(x,y)\big)\\
                &\geq \frac 13 ( d_X(z, y)+ d_X(x,y) ) \\
                & \geq \frac 13 \left( \frac {1}{\qq} |t_2 - t_1| - \sQ+ |t_1 - t_0| \right)
                   \geq \frac{1}{3\qq} |t_2 -t_0| - \frac{\sQ}3.\\      
\end{align*}
This established the lower-bound. The upper-bound follows from the triangle inequality:
\[
d_X(x,z) \leq d_X(x,y)+ d_X(y,z) \leq |t_1-t_0| + \qq |t_2-t_1| + \sQ
\leq \qq |t_2 - t_0| + \sQ. 
\]
It follows that $\gamma$ is a $(3\qq, \sQ)$--quasi-geodesic. 
\end{proof}

\subsection*{The boundaries of \CAT spaces}
A proper \CAT space $X$ can be compactified via the \textit{visual boundary}. 
The points of the visual boundary $\partial_\infty X$ of $X$ are geodesic rays (starting 
from $\go$). Set $\overline{X} = X \bigcup \partial_{\infty} X$ where points in 
$\overline X$ can be thought of as geodesic rays or geodesic segments starting from 
$\go$. The space $\overline{X}$ is usually equipped with the \textit{cone topology}
where two geodesics are considered nearby if they fellow travel each other for a long time
(see  \cite{CAT(0)reference} for more details). 

\section{The $\kappa$--Morse geodesics of $X$}
The goal of this section is to prove \thmref{Thm:TFAE} which gives 
several equivalent characterizations of the notion of a $\kappa$--Morse geodesic
(or quasi-geodesic) ray. 

\subsection{Sublinear functions}
We fix a function 
\[
\kappa \from [0,\infty) \to [1,\infty)
\] 
that is monotone increasing, concave and sublinear, that is
\[
\lim_{t \to \infty} \frac{\kappa(t)} t = 0. 
\]
Note that using concavity, for any $a>1$, we have
\begin{equation} \label{Eq:Concave}
\kappa(a t) \leq a \left( \frac 1a \, \kappa (a t) + \left(1- \frac 1a\right) \kappa(0) \right) 
\leq a \, \kappa(t).
\end{equation}

We say a quantity $\sD$ \emph{is small compared to a radius $\rr>0$} if 
\begin{equation} \label{Eq:Small} 
\sD \leq \frac{\rr}{2\kappa(\rr)}. 
\end{equation}

\begin{remark}
The assumption that $\kappa$ is increasing and concave makes certain arguments
cleaner, otherwise they are not really needed. One can always replace any 
sublinear function $\kappa$, with another sublinear function $\overline \kappa$
so that $\kappa(t) \leq \overline \kappa(t) \leq \sC \, \kappa(t)$ for some constant $\sC$ 
and $\overline \kappa$ is monotone increasing and concave. For example, define 
\[
\overline \kappa(t) = \sup \Big\{ \lambda \kappa(u) + (1-\lambda) \kappa(v) \ST 
\ 0 \leq \lambda \leq 1, \ u,v>0, \ \text{and}\ \lambda u + (1-\lambda)v =t \Big\}.
\]
The requirement $\kappa(t) \geq 1$ is there to remove additive errors in the definition
of $\kappa$--Morse geodesics. 
\end{remark}

\begin{lemma}\label{Lem:sublinear-estimate}
For any $\sD_0>0$, there exists  $\sD_{1}, \sD_{2} > 0$ depending on $\sD_0$
and $\kappa$ so that, for $x, y \in X$,
\[
d(x, y) \leq \sD_{0} \cdot \kappa(x)
\qquad\text{implies}\qquad
\sD_{1} \kappa(x) \leq \kappa(y) \leq \sD_{2} \kappa(x).
\]
\end{lemma}
\begin{proof}
Since $\kappa$ is sublinear, there is a constant $\sA$ such that, for every $u>0$, 
\[
\kappa(u) \leq \frac{u}{2 \sD_0} + \sA. 
\]
For $x \in X$, define $\Norm{x} = d_X(\go, x)$. Then 
\begin{equation} \label{Eq:Norm-Difference}
\Big| \Norm x - \Norm y \Big| 
\leq d_X(x, y) 
\leq \sD_{0} \cdot \kappa(x) 
\leq \sD_{0} \cdot \left(\frac{\Norm x}{2 \sD_0} + \sA\right) 
\leq \frac 12 \Norm x + \sD_0 \sA.                          
\end{equation} 
We argue in two cases. Suppose $\Norm x \geq \Norm y$. Then, 
\eqnref{Eq:Norm-Difference} implies
\[
  \Norm x \leq 2 \Norm y + 2\sD_0 \sA,
\]
and from \eqnref{Eq:Concave}, we get 
\[
  \kappa(x) \leq (2+2\sD_0 \sA )\cdot\kappa(y).
\]                              
Thus 
\[
 (2+2\sD_0 \sA)^{-1}\kappa(x) \leq \kappa(y) \leq \kappa(x).
\]
On the other hand, if $\Norm x < \Norm y$, then \eqnref{Eq:Norm-Difference} implies
\[
\Norm y \leq \frac 32 \Norm x +  \sD_0 \sA.
\]
Again, by \eqnref{Eq:Concave} we have
\[
 \kappa(y) \leq \left(\frac 32 + \sD_0 \sA \right) \cdot \kappa(x)
\] 
and hence
\[
\kappa(x) < \kappa(y) \leq  \left(\frac 32+\sD_0 \sA\right) \cdot \kappa(x).
\]
Combining the two cases, we get
\[ 
(2+2\sD_0 \sA)^{-1}\kappa(x) \leq \kappa(y) 
   \leq \left(\frac 32+\sD_0 \sA\right) \cdot\kappa(x).
\]
That is, the lemma holds for $\sD_{1} = (2+2\sD_0 \sA)^{-1}$ and 
$\sD_{2} =  \frac 32+\sD_0 \sA $.
\end{proof}

\begin{definition}[$\kappa$--neighborhood]  \label{Def:Neighborhood} 
For a closed set $Z$ and a constant $\nn$ define the $(\kappa, \nn)$--neighbourhood 
of $Z$ to be 
\[
\calN_\kappa(Z, \nn) = \Big\{ x \in X \ST 
  d_X(x, Z) \leq  \nn \cdot \kappa(x)  \Big\}.
\]

\end{definition} 

\begin{figure}[h!]
\begin{tikzpicture}
 \tikzstyle{vertex} =[circle,draw,fill=black,thick, inner sep=0pt,minimum size=.5 mm] 
[thick, 
    scale=1,
    vertex/.style={circle,draw,fill=black,thick,
                   inner sep=0pt,minimum size= .5 mm},
                  
      trans/.style={thick,->, shorten >=6pt,shorten <=6pt,>=stealth},
   ]

 \node[vertex] (a) at (0,0) {};
 \node at (-0.2,0) {$\go$};
 \node (b) at (10, 0) {};
 \node at (10.6, 0) {$b$};
 \node (c) at (6.7, 2) {};
 \node[vertex] (d) at (6.68,2) {};
 \node at (6.7, 2.4){$x$};
 \node[vertex] (e) at (6.68,0) {};
 \node at (6.7, -0.5){$x_{b}$};
 \draw [-,dashed](d)--(e);
 \draw [-,dashed](a)--(d);
 \draw [decorate,decoration={brace,amplitude=10pt},xshift=0pt,yshift=0pt]
  (6.7,2) -- (6.7,0)  node [black,midway,xshift=0pt,yshift=0pt] {};

 \node at (7.8, 1.2){$\nn \cdot \kappa(x)$};
 \node at (3.6, 0.7){$||x||$};
 \draw [thick, ->](a)--(b);
 \path[thick, densely dotted](0,0.5) edge [bend left=12] (c);
\node at (1.4, 1.9){$(\kappa, \nn)$--neighbourhood of $b$};
\end{tikzpicture}
\caption{The $(\kappa,n)$--neighbourhood of the geodesic ray $b$.}
\end{figure}
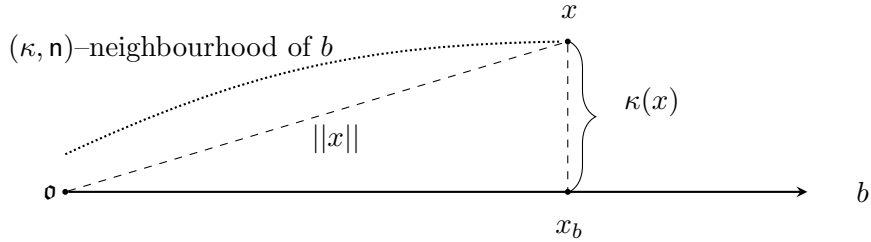

In view of \remref{Rem:Projection}, a geodesic or a quasi-geodesic can take the place 
of the set $Z$ in the above definitions. That is, we can write 
$\calN_{\kappa}(b, \nn)$ to mean the $(\kappa, \nn)$--neighborhood of the image of 
the geodesic ray $b$. Or, we can use phrases like 
``the quasi-geodesic $\beta$ is $\kappa$--contracting" or 
``the geodesic $b$ is in a $(\kappa, \nn)$--neighbourhood of the geodesic $c$". 

\begin{definition} \label{Def:Fellow-Travel}
Let $\beta$ and $\gamma$ be two quasi-geodesic rays in $X$. If $\beta$ is in some 
$\kappa$--neighbourhood of $\gamma$ and $\gamma$ is in some 
$\kappa$--neighbourhood of $\beta$, we say that $\beta$ and $\gamma$ 
\emph{$\kappa$--fellow travel} each other. This defines an equivalence
relation on the set of quasi-geodesic rays in $X$ (to obtain transitivity, one needs to change 
$\nn$ of the associated $(\kappa, \nn)$--neighbourhood). We refer to such an equivalence
class as a \emph{$\kappa$--equivalence class of quasi-geodesics}.
We denote the $\kappa$--equivalence class that contains $\beta$ by $[\beta]$ or we use 
the notation $\bfb$ for such an equivalence
class when no quasi-geodesic in the class is given. 
\end{definition}

\begin{lemma} \label{Lem:Unique}
Let $b \from [0,\infty) \to X$ be a geodesic ray in $X$. Then $b$ is the unique geodesic 
ray in any $(\kappa, \nn)$--neighbourhood of $b$ for any $\nn$. That is to say, distinct geodesic 
rays do not $\kappa$--fellow travel each other.
\end{lemma}

\begin{proof}
Consider any other geodesic ray $c \from [0,\infty) \to X$ emanating from the same base-point. 
Then, there is a time $t_0$ where $b(t_0) \not = c(t_0)$. For a given $t \geq t_0$, let $t'$ be the time 
such that 
\[
d_X(c(t), b) = d_X\big(c(t), b(t') \big).  
\]
That is, $b(t')$ is the projection of $c(t)$ to $b$. Since $X$ is a \CAT space, we have 
\[
d_X\big(c(t), b(t')\big) \geq 
  \frac{t}{t_0} \cdot d_X\left( c(t_0), b\left(\frac{t' \, t_0}t \right)\right)
  \geq \frac{d_X\big(c(t_0), b\big)}{t_0} \cdot t.
\]
This means that the distance from $c(t)$ to $b$ grows linearly with $t$ and hence $c$ 
is not contained in any $(\kappa, \nn)$--neighborhood of $b$. 
\end{proof}

\subsection{$\kappa$--Morse and $\kappa$--contracting sets}

\begin{definition}[weakly $\kappa$--Morse] \label{Def:W-Morse} 
We say a closed subset $Z$ of $X$ is \emph{weakly $\kappa$--Morse} if there is a function
\[
\mm_Z \from \RR_+^2 \to \RR_+
\]
so that if $\beta \from [s,t] \to X$ is a $(\qq, \sQ)$--quasi-geodesic with end points 
on $Z$ then
\[
\beta[s,t]  \subset \calN_{\kappa} \big(Z,  \mm_Z(\qq, \sQ)\big). 
\]
We refer to $\mm_{Z}$ as the \emph{Morse gauge} for $Z$. We always assume
\begin{equation}
\mm_Z(\qq, \sQ) \geq \max(\qq, \sQ). 
\end{equation} 
\end{definition}

\begin{definition}[strongly $\kappa$--Morse] \label{Def:S-Morse} 
We say a closed subset $Z$ of $X$ is \emph{strongly $\kappa$--Morse} if there is a function 
$\mm_Z\from \RR^2 \to \RR$ such that, for every constants $\rr>0$, $\nn>0$ and every
sublinear function $\kappa'$, there is an $\sR= \sR(Z, \rr, \nn, \kappa')>0$ where the 
following holds: Let $\eta \from [0, \infty) \to X$ be a $(\qq, \sQ)$--quasi-geodesic ray 
so that $\mm_Z(\qq, \sQ)$ is small compared to $\rr$, let $t_\rr$ be the first time 
$\Norm{\eta(t_\rr)} = \rr$ and let $t_\sR$ be the first time $\Norm{\eta(t_\sR)} = \sR$. Then
\[
d_X\big(\eta(t_\sR), Z\big) \leq \nn \cdot \kappa'(\sR)
\quad\Longrightarrow\quad
\eta[0, t_\rr] \subset \calN_{\kappa}\big(Z, \mm_Z(\qq, \sQ)\big). 
\]
\end{definition} 

\begin{remark}
Colloquially, the strongly Morse condition can be stated as saying that if $\eta$ is in a sublinear 
neighborhood of $Z$ for any sublinear function $\kappa'$ then, in fact, it is contained in a 
$\kappa$--neighborhood of $Z$. That is, sublinear fellow traveling implies uniform sublinear fellow 
traveling. This is a natural generalization of the notion of a Morse set which can be stated as fellow 
traveling implies uniform fellow traveling.  
\end{remark}

\begin{definition}[$\kappa$--contracting] \label{Def:Contracting}
Recall that, for $x \in X$, we have $\Norm{x} = d_X(\go, x)$. 
For a closed subspace $Z$ of $X$, we say $Z$ is \emph{$\kappa$--contracting} if there 
is a constant $\cc_Z$ so that, for every $x,y \in X$
\[
d_X(x, y) \leq d_X( x, Z) \quad \Longrightarrow \quad
\diam_X \big( x_Z \cup y_Z \big) \leq \cc_Z \cdot \kappa(\Norm x).
\]
To simplify notation, we often drop $\Norm{\param}$. That is, for $x \in X$, we define
\[
\kappa(x) := \kappa(\Norm{x}). 
\]
\end{definition} 

\begin{theorem} \label{Thm:TFAE}
Let $\bfb$ be a $\kappa$--equivalence class of quasi-geodesics in $X$. 
The following properties of $\bfb$ are equivalent. 
\begin{enumerate}
\item The class $\bfb$ contains a geodesic ray $b$ that is $\kappa$--contracting. 
\item Every quasi-geodesic $\beta \in \bfb$ is $\kappa$--contracting. 
\item Every quasi-geodesic $\beta \in \bfb$ is strongly $\kappa$--Morse. 
\item There exists a quasi-geodesic $\beta \in \bfb$ that is strongly $\kappa$--Morse. 
\item Every quasi-geodesic $\beta \in \bfb$ is weakly $\kappa$--Morse. 
\item There exists a quasi-geodesic $\beta \in \bfb$ that is weakly $\kappa$--Morse. 
\item The class $\bfb$ contains a geodesic ray $b$ that is weakly $\kappa$--Morse for 
$(32,0)$--quasi-geodesics. 
\end{enumerate}
\end{theorem}

Note that the implications $(3) \Longrightarrow (4)$ and $(5) \Longrightarrow (6)$ are immediate. 
Later in this section, we will prove $(6) \Longrightarrow (7) \Longrightarrow (1)  \Longrightarrow (2)  
\Longrightarrow (3) \Longrightarrow (5)$ in separate statements. To prepare for the first statement, 
we study the finite geodesic segments connecting points of the $\kappa$--Morse 
quasi-geodesic.

\begin{proposition}\label{Morsegeodesic}
Let  $X$ be a proper CAT(0) space. Let $\beta \from [0, \infty) \to X$ be a $(\qq, \sQ)$--quasi-geodesic 
ray in $X$ that is $\kappa$--Morse with $\mm_\beta$ as its Morse gauge. For any given 
$T \in (0, \infty)$, let $b=b_T$ be the finite geodesic segment 
connecting $\beta(0)=\go$ and $\beta(T)$. Then $b$ is $\kappa$--Morse 
and the Morse-gauge of $b$ is independent of $T$. 
That is, there exists $\mm \from \RR^2 \to \RR$ such that for every $T \in [0, \infty)$
and for every $(\qq', \sQ')$--quasi-geodesic $\zeta \from [s,t] \to X$ with end points in 
$b=b_T$, we have
\[
\zeta[s,t] \subset \calN_{\kappa} \big(b, \mm(\qq', \sQ') \big).
\]
\end{proposition}

\begin{proof}
We parametrize $b \from [0, \dd] \to X$ by arc length so $\dd=d_X(\beta(0), \beta(T))$. 
The geodesic segment $b$ can be considered as a $(1,0)$--quasi-geodesic with end points 
on $\beta$. Hence, for every $0 \leq s \leq \dd$, there is $t_s \in [0, \infty)$ so that 
\begin{equation} \label{Eq:t_s}
d_X\big(b(s) , \beta(t_s) \big) \leq \mm_\beta(1,0) \cdot \kappa(s). 
\end{equation}
We take $t_0 = 0$ and $t_d = T$. We show that $\beta[0,T]$ stays in some uniform 
$\kappa$--neighborhood of $b$ by arguing that the times $t_s$ nearly cover the 
interval $[0,T]$. Let $0=s_0, s_1, \dots, s_k=d$ be a set of times so that 
$|s_i - s_{i+1}| \leq 1$.  Then, for every $t \in [0,T]$ we have $t_{s_0} \leq t \leq t_{s_k}$. 
Hence there is an index $i$ such that $t_{s_{i-1}}\leq t$ and $t_{s_{i}} \geq t$ . 

\begin{figure}[ht]
\begin{tikzpicture}[scale=0.9]
 \tikzstyle{vertex} =[circle,draw,fill=black,thick, inner sep=0pt,minimum size=.5 mm]
 
[thick, 
    scale=1,
    vertex/.style={circle,draw,fill=black,thick,
                   inner sep=0pt,minimum size= .5 mm},
                  
      trans/.style={thick,->, shorten >=6pt,shorten <=6pt,>=stealth},
   ]

  \node[vertex] (a) at (-5,0) [label=180:$\go$] {}; 
  \node[vertex] (b) at (5,0) [label=-10:$\beta(T)$] {}; 
  \node[vertex] (c)  at(1,.9)  [label=90:$\beta(t_{s_{i-1}})$] {}; 
  \node[vertex] at(1, 0)  [label=below:$b(s_{i-1})$] {}; 
  \node[vertex] (d)  at (4,1.5)  [label=above right:$\beta(t_{s_i})$] {}; 
  \node[vertex] at(4, 0)  [label=below:$b(s_i)$] {}; 
  \node[vertex] at(3,1.9)  [label=above:$\beta(t)$] {}; 
  \draw [thick, blue] (a) -- (b);
  \draw[thick, red, dashed] (1, 0)--(c){};
  \draw[thick, red, dashed](4, 0) --(d){};
  \node at (-4.3,.1) [label=below:$b$]{};
  \node at (-4.2,.8) [label=above:$\beta$]{};
 
  \pgfsetlinewidth{1pt}
  \pgfsetplottension{.75}
  \pgfplothandlercurveto
  \pgfplotstreamstart
  \pgfplotstreampoint{\pgfpoint{-5cm}{0cm}}  
  \pgfplotstreampoint{\pgfpoint{-4cm}{1cm}}   
  \pgfplotstreampoint{\pgfpoint{-3cm}{.-.5cm}}
  \pgfplotstreampoint{\pgfpoint{-1cm}{.7cm}}
  \pgfplotstreampoint{\pgfpoint{-2cm}{1.2cm}}
  \pgfplotstreampoint{\pgfpoint{-1cm}{2cm}}
  \pgfplotstreampoint{\pgfpoint{1cm}{.9cm}}
  \pgfplotstreampoint{\pgfpoint{3cm}{1.9cm}}
  \pgfplotstreampoint{\pgfpoint{4cm}{1.5cm}}
  \pgfplotstreampoint{\pgfpoint{6cm}{1cm}}
  \pgfplotstreampoint{\pgfpoint{5cm}{0cm}}
  \pgfplotstreamend
  \pgfusepath{stroke}
 
\end{tikzpicture}
\caption{The index $i$ is chosen so that $t_{s_{i-1}} \leq t \leq t_{s_i}$.} 
 \end{figure}
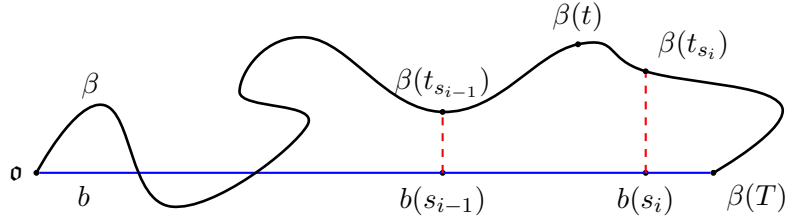

We have 
\begin{align*}
d_X(\beta(t_{s_{i-1}}), \beta(t_{s_{i}})) 
&\leq d_X(\beta(t_{s_{i-1}}), b(s_{i-1})) +d_X(b(s_{i-1}), b(s_{i})) 
     + d_X(b(s_{i}), \beta(t_{s_{i}})) \\
&\leq  \mm_\beta(1,0) \cdot \kappa(s_{i-1}) + 1 + \mm_\beta(1,0) \cdot \kappa(s_{i}).  
\end{align*}
Using the lower-bound condition for a $(\qq, \sQ)$--quasi-geodesic we have
\[
|t_{s_{i}}-t_{s_{i-1}}| \leq \qq d_X(\beta(t_{s_{i-1}}), \beta(t_{s_{i}})) + \qq \sQ
\leq \qq \big(2 \mm_\beta(1,0) \kappa(s_{i}) +1\big) + \qq \sQ.  
\]
From this and using the upper-bound condition, we get  
\begin{align*}
d_X\big(\beta(t_{s_i}), \beta(t)\big) &\leq \qq |t_{s_i} - t| + \sQ \\
 &\leq \qq |t_{s_i} - t_{s_{i-1}}| + \sQ \\
 &\leq \qq^2 (2 \mm_\beta(1,0) \, \kappa(s_{i}) +1) + \qq^2 \sQ + \sQ. 
\end{align*}
Combining this with \eqnref{Eq:t_s}, we get that there is a function 
$\mm_1 \from \RR^2 \to \RR$ depending only on the value of $\mm_\beta(1,0)$ so that 
\begin{equation} \label{Eq:m_1}
d_X(\beta(t), b(s_i)) \leq d_X(\beta(t), \beta(t_{s_i})) + d_X(\beta(t_{s_i})), b(s_{i})) \leq \mm_1(\qq, \sQ) \cdot \kappa(s_i). 
\end{equation} 
By Lemma~\ref{Lem:sublinear-estimate}, there exists $\mm_{2}$ depending only on 
$\mm_1(\qq, \sQ)$ and $\kappa$ such that 
\[
\kappa(s_{i}) = \kappa(b(s_i)) \leq \mm_{2}\cdot \kappa(\beta(t)).
\]
Thus we have 
\begin{equation} \label{Eq:m_2}
\beta[0,T] \subset \calN_{\kappa} \big(b, \mm_2(\qq, \sQ)\big). 
\end{equation}

Now consider a $(\qq', \sQ')$--quasi-geodesic $\zeta \from [s,t] \to X$ with end points on 
$b$. To show that $\zeta$ stays near $b$, we modify $\zeta$ to a 
$(9\qq', \sQ')$--quasi-geodesic $\zeta'$ with end points on $\beta$ which implies that 
$\zeta'$ stays near $\beta$ since $\beta$ is $\kappa$--Morse. The \eqnref{Eq:m_2} 
then implies that $\zeta$ stays near $b$ as well. 

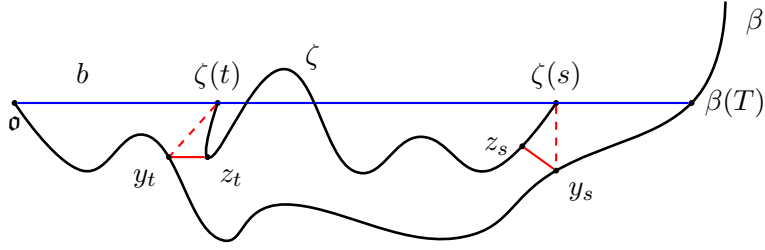
\begin{figure}[ht]
\begin{tikzpicture}[scale=0.9]
 \tikzstyle{vertex} =[circle,draw,fill=black,thick, inner sep=0pt,minimum size=.5 mm]
 
[thick, 
    scale=1,
    vertex/.style={circle,draw,fill=black,thick,
                   inner sep=0pt,minimum size= .5 mm},
                  
      trans/.style={thick,->, shorten >=6pt,shorten <=6pt,>=stealth},
   ]

  \node[vertex] (a) at (-5,0)  [label=below:$\go$] {}; 
  \node[vertex] (b) at (5,0)   [label=right:$\beta(T)$] {}; 
  \node[vertex] (xb) at (-2,0)  [label=above:$\zeta(t)$] {}; 
  \node[vertex] (xa) at (3, 0)  [label=above:$\zeta(s)$] {}; 
  \node[vertex] (za) at (2.5, -.64) [label=left:$z_s$]{};
  \node[vertex] (ya) at (3, -1) [label=below right:$y_s$]{};
  \node[vertex] (zb) at (-2.15, -.8) [label=below right:$z_t$]{};
  \node[vertex] (yb) at (-2.72, -.8) [label=below left:$y_t$]{};
  \draw [thick, blue] (a) -- (b);
  \node at (-4,0) [label=above:$b$]{};
  \node at (5.5,1.2) [label=right:$\beta$]{};
  \node at (-.6,.2) [label=above:$\zeta$]{};

  \pgfsetlinewidth{1pt}
  \pgfsetplottension{.75}
  \pgfplothandlercurveto
  \pgfplotstreamstart
  \pgfplotstreampoint{\pgfpoint{-5cm}{0cm}}  
  \pgfplotstreampoint{\pgfpoint{-4cm}{-1cm}}   
  \pgfplotstreampoint{\pgfpoint{-3cm}{-.5cm}}
  \pgfplotstreampoint{\pgfpoint{-2cm}{-2cm}}
  \pgfplotstreampoint{\pgfpoint{-1cm}{-1.5cm}}
  \pgfplotstreampoint{\pgfpoint{1.5cm}{-2cm}}
  \pgfplotstreampoint{\pgfpoint{3cm}{-1cm}}
  \pgfplotstreampoint{\pgfpoint{5cm}{0cm}}
  \pgfplotstreampoint{\pgfpoint{5.5cm}{1.5cm}}
  \pgfplotstreamend
  \pgfusepath{stroke}    
 
  \pgfsetplottension{.75}
  \pgfplothandlercurveto
  \pgfplotstreamstart
  \pgfplotstreampoint{\pgfpoint{-2cm}{0cm}}  
  \pgfplotstreampoint{\pgfpoint{-2.1cm}{-.8cm}}  
  \pgfplotstreampoint{\pgfpoint{-1cm}{.5cm}}  
  \pgfplotstreampoint{\pgfpoint{0cm}{-1cm}}   
  \pgfplotstreampoint{\pgfpoint{1cm}{-.5cm}}
  \pgfplotstreampoint{\pgfpoint{2cm}{-1cm}}
  \pgfplotstreampoint{\pgfpoint{3cm}{0cm}}
  \pgfplotstreamend
  \pgfusepath{stroke}

  \draw[thick, red, dashed] (xa)--(ya){};
  \draw[thick, red] (ya)--(za){}; 
  \draw[thick, red, dashed] (xb)--(yb){};
  \draw[thick, red] (yb)--(zb){}; 
\end{tikzpicture}
\caption{The concatenation of $[y_s, z_s]$, $[z_s, z_t]_\zeta$ and $[z_t, y_t]$ is
a quasi-geodesic with end points on $\beta$.}
 \end{figure}

Let $y_s\in \beta$ be the closest-point in $\beta$ to $\zeta(s)$ and let $z_s$ be the closest
point in $\zeta$ to $y_s$. By \lemref{concatenation} the concatenation of the 
geodesic segment $[y_s, z_s]$ and the quasi-geodesic segment $[z_s, \zeta(t)]_{\zeta}$ 
forms a $(3\qq', \sQ')$--quasi-geodesic. Similarly we can find points $y_t \in \beta$ and 
$z_t \in \zeta$ and apply \lemref{concatenation} again. Denote the concatenation of the 
geodesic segment $[y_s, z_s]$, the quasi-geodesic segment $[z_s, z_t]_{\zeta}$ and the 
geodesic segment $[z_t, y_t]$ by $\zeta'$ which is a $(9\qq', \sQ')$--quasi-geodesic. Then 
\begin{equation} \label{Eq:zeta'}
\zeta' \subset \calN_{\kappa}\big(\beta, \mm_\beta(9\qq', \sQ')\big).
\end{equation} 

We say $x$ is $\kappa$--close to $y$, if there is a constant $\cc$ depending on 
$\qq, \sQ, \qq', \sQ'$ and $\mm_\beta$ such that $d_X(x,y) \leq \cc \cdot \kappa(x)$.
It follows from Lemma~\ref{Lem:sublinear-estimate} that if $x$ is $\kappa$--close to 
$y$ and $y$ is $\kappa$--close to $z$ then $x$ is $\kappa$--close $z$. 
Thus every point in $\zeta$ is $\kappa$--close to a point in $\zeta'$. Now 
\eqnref{Eq:zeta'} and \eqnref{Eq:m_2} imply that 
\[
\zeta \subset \calN_\kappa\big(b, \mm(\qq', \sQ')\big)
\]
for some $\mm \from \RR^2 \to \RR$ depending on $\qq, \sQ$ and $\mm_\beta$ only. 
\end{proof}

\begin{proposition}[$(6) \Longrightarrow (7)$]  \label{unique}
If $\beta \from [0, \infty) \to X$ is a $\kappa$--Morse quasi-geodesic ray then 
\begin{enumerate}
 \item the class $\bfb=[\beta]$ contains a geodesic $b$, and 
 \item the geodesic $b$ is $\kappa$--Morse (in particular, for $(32,0)$--quasi-geodesics).
\end{enumerate}
\end{proposition}

\begin{proof}
For $n \in \NN$, let $b_n$ be the geodesic segment connecting $\go$ to
$\beta(n)$. Up to taking a subsequence, we can assume the geodesic segments 
$b_n$ converge to a geodesic ray $b$ in $X$. Since $\beta$
is $\kappa$--Morse, $b_n \subset \calN_{\kappa}\big(\beta, \mm_\beta(1,0)\big)$ which means
$b \subset \calN_{\kappa}\big(\beta, \mm_\beta(1,0)\big)$. That is, $b \in [\beta]$. 
But the class $[\beta]$ contains only one geodesic (\lemref{Lem:Unique}) hence
any other subsequence of $b_n$ has to also converge to $b$. 
In particular, every point in $b$ is the limit of points in $b_n$ and
every limit point of a sequence $x_n \in b_n$ is on $b$. 

The second part follows almost immediately from \propref{Morsegeodesic}. 
For every quasi-geodesic $\zeta$ with end points on $b$, there is 
$n_0$ so that for $n \geq n_0$, the end points of $\zeta$ are distance 1 from some point 
in $b_n$. Then $\zeta$ can be modified slightly to have end points
in $b_n$. \propref{Morsegeodesic} implies that $\zeta$ stays in a $\kappa$--neighborhood 
of $b_n$. But this is true for every $n \geq n_0$. Hence
$\zeta$ stays in some $\kappa$--neighborhood of $b$. 
\end{proof}

To prepare for the next step, we recall a construction of quasi-geodesics from 
\cite{contracting}. 

\begin{proposition}[\cite{contracting}]\label{charneysultan}
Given a geodesic segment (possibly infinite) $b$ and points $x, y \in X$  such that $d_X(x, y) < d_X(x, b)$, 
there exists a $(32, 0)$--quasi-geodesic $\zeta \from [s_0, s_1] \to X$ with endpoints 
on $b$ such that $\zeta(s_0)= x_{b}$, 
\[
\frac 14 d_X(x_{b}, y_{b}) \leq d_X(\zeta(s_0), \zeta(s_1)) < d_X(x_{b}, y_{b}) 
\] 
and there is a point $p = \zeta(t)$ on $\zeta$ so that 
\begin{equation}\label{case1}
d_X(p, b) \geq \frac{1}{80} d_X(x_b, y_b).
\end{equation}
\end{proposition}
\begin{proof}[Outline of the proof of Proposition~\ref{charneysultan}]
The proof of this statement is contained in the proof of  \cite[ Theorem 2.9]{contracting}.
We now give 
the outline of the argument and a detailed reference to that proof. Given a geodesic
$b$ and points $x$ and $y$ that satisfy the assumptions, consider the following quadrilateral:
\[
Q_{1} = [x, x_{b}] \cup [x_{b}, y_{b}] \cup [y, y_{b}] \cup [x, y].
\]
We first construct a smaller quadrilateral inside $Q_{1}$ out of two points 
$x', y'$ where $x'$ on the segments $[x, x_{b}]$ and $y'$ is either in the interior of 
the geodesic segment connecting $x$ to $y$ (Theorem 2.9, Case (2)) or on 
$[y, y_{b}]$ (Theorem 2.9, Case (1)) and consider the quadrilateral
\[
Q_{2} = [x', x_{b}] \cup [x_{b}, y_{b}] \cup [y', y_{b}] \cup [x', y']
\]
with the property (in all cases) that 
\[
d_X(x'_{b}, y'_{b}) \geq \frac 14 d_X(x_{b}, y_{b}).
\]
Let $\sD = d_X(x'_{b}, y'_{b})$, and let $\aa, \bb, \cc>0$ be real numbers such that 
\begin{align*}
d_X(x', x'_{b}) &= \aa \, \sD\\
d_X(x', y') &= \bb \, \sD\\
d_X(y', y'_{b}) &= \cc \, \sD
\end{align*}
The quadrilateral $Q_{2}$ also satisfies the condition that $\aa+\cc -\bb > 0.1$ 
and $\aa+\bb+\cc < 8$ (worst case is Case (1); in Case (3) it is shown that 
$\aa+\cc -\bb > 0.2$).

Next we construct a quasi-geodesics $\zeta(t)$ that starts from $x'_{b}$ follows along 
the segment $[x'_b, x']$ until it is close to the segment $[x', y']$, then travels to $[x', y']$ 
and follows $[x', y']$  until it is close to $[y'_b, y']$, then it travels to  $[y'_b, y']$ and finally 
follows $[y'_b, y']$ until $y'_b$.  \cite[Lemma 2.7]{contracting} establishes that $\zeta(t)$ 
is a $(4(\aa+\bb+\cc), 0)$--quasi-geodesic, that is, $\zeta$ is a $(32,0)$--quasi-geodesic. 
Let $p$ be a point on $\zeta(t)$ on the segment between $x'$ and $y'$. Equation (4) of  \cite{contracting} 
states that 
\[ d_{X}(p, b) \geq \frac{\aa+\cc -\bb}{2} \sD.\]

Combining this with $\aa+\cc -\bb > 0.1$ we have 
\[
d_{X}(p, b) \geq \frac{1}{20} d_{X}(x'_{b}, y'_{b})
  \geq \frac 1{80} d_X(x_b, y_b). 
\]
This finishes the proof.
\end{proof} 

\begin{theorem}[$(7) \Longrightarrow (1)$] \label{Morseimpliescontracting}
Let $b$ be a geodesic ray in $X$  that is $\kappa$--Morse for 
$(32,0)$--quasi-geodesics. Then $b$ is $\kappa$--contracting.
In fact, $\cc_b = 82000\, \mm_b(32, 0)$.
\end{theorem}

\begin{proof}
Given points $x, y$ such that $d_X(x, y) < d_X(x, b)$ let 
$\zeta \from [s_0,s_1] \to X$ and $p= \zeta(t)$ 
be as in \propref{charneysultan}. Since $b$ is $\kappa$--Morse for
$(32,0)$--quasi-geodesics, we have 
\[
d_X(p, b)\leq \mm_b(32, 0) \cdot \kappa(p).
\]
On the other hand, 
\begin{align*}
\Norm p &\leq \Norm {x_{b}} + d_X\big(\zeta(s_{0}), \zeta(t)\big)\\
              &\leq \Norm {x_{b}} + 32 \cdot |s_{1}-s_{0}| 
                     &&\text{$\zeta$ is a (32, 0)--quasi-geodesic}\\
              &\leq \Norm {x_{b}} + (32)^2 \cdot d_X\big(\zeta(s_{0}), \zeta(s_{1})\big)
                     &&\text{$\zeta$ is a (32, 0)--quasi-geodesic}\\
              &\leq \Norm {x_{b}}+ 1024 \cdot d_X(x_{b}, y_{b}) \\
              &\leq\Norm {x_{b}} + 1024 \cdot d_X(x, y) 
                    &&\text{Projection to $b$ is Lipschitz.}\\
              &\leq \Norm {x_{b}}+ 1024 \cdot d_X(x, x_{b}) \\
              &\leq 1025 \cdot  \Norm {x_{b}}.
\end{align*}
Therefore,
\begin{align*} 
d_X(x_{b}, y_{b})  &\leq 80 \cdot d_X(p, b)\\
   & \leq 80 \cdot \mm(32, 0) \cdot \kappa(p)\\
    & \leq 80 \cdot \mm(32, 0) \cdot \kappa(1025 \Norm x)\\
    & \leq 82000 \cdot \mm(32, 0) \cdot \kappa(x).
\end{align*} 
That is, $b$ is a $\kappa$--contracting geodesic with $\cc_b= 82000 \cdot \mm_b(32, 0)$.
\end{proof}

\begin{proposition}[$(1) \Longrightarrow (2)$] 
Le $b$ be a geodesic ray and let $\beta$ be a quasi-geodesic ray in $\bfb=[b]$. 
Suppose that $b$ is $\kappa$--contracting. Then $\beta$ is also $\kappa$--contracting. 
\end{proposition}

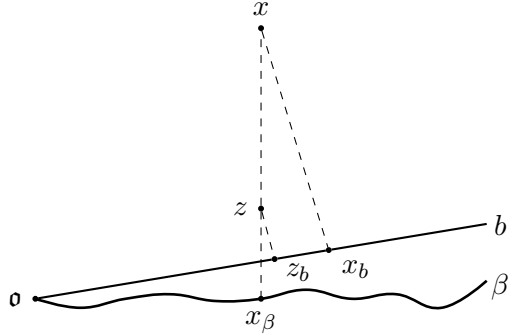
\begin{figure}[ht]
\begin{tikzpicture}[scale=0.6]
 \tikzstyle{vertex} =[circle,draw,fill=black,thick, inner sep=0pt,minimum size=.5 mm]
 
[thick, 
    scale=1,
    vertex/.style={circle,draw,fill=black,thick,
                   inner sep=0pt,minimum size= .5 mm},
                  
      trans/.style={thick,->, shorten >=6pt,shorten <=6pt,>=stealth},
   ]

  \node[vertex] (o) at (-5,0)  [label=left:$\go$] {}; 
  \node[vertex] (p) at (0,0)  [label=below:$x_\beta$] {};
  \node (xi) at (5.3,1) [label=below:$\beta$] {}; 
  \node (g) at (5.3,2.3) [label=below:$b$] {}; 

  \node [vertex] (c) at (.3,.88) {};
  \node at (.1,1.05)  [label=-10:$z_b$] {};
  \node [vertex] (d) at (1.5, 1.08)  [label=-10:$x_b$] {};
  \draw[thick]  (o)--(c)--(d)--(5,1.66){};
  \node[vertex] (b) at (0,2)  [label=left:$z$] {}; 
  \node[vertex] (x) at (0,6)  [label=above:$x$] {}; 
  \draw [dashed](0,0)--(b)--(x){};
  \draw [dashed](x)--(d){};
  \draw [dashed](b)--(c){};
       
  \pgfsetlinewidth{1pt}
  \pgfsetplottension{.75}
  \pgfplothandlercurveto
  \pgfplotstreamstart
  \pgfplotstreampoint{\pgfpoint{-5cm}{0cm}}  
  \pgfplotstreampoint{\pgfpoint{-4cm}{-.2cm}}   
  \pgfplotstreampoint{\pgfpoint{-3cm}{0cm}}
  \pgfplotstreampoint{\pgfpoint{-2cm}{.1cm}}
  \pgfplotstreampoint{\pgfpoint{-1cm}{-.05cm}}
  \pgfplotstreampoint{\pgfpoint{0cm}{0cm}}
  \pgfplotstreampoint{\pgfpoint{1cm}{.2cm}}
  \pgfplotstreampoint{\pgfpoint{2cm}{0cm}}
  \pgfplotstreampoint{\pgfpoint{3cm}{.1cm}}
  \pgfplotstreampoint{\pgfpoint{4cm}{-.2cm}}
  \pgfplotstreampoint{\pgfpoint{5cm}{.4cm}}
  \pgfplotstreamend 
  \pgfusepath{stroke}
       
  \end{tikzpicture}
 \end{figure}

\begin{proof}
Since $\beta$ and $b$ are in the same class, there exists $\nn$ such that 
\[
\beta \subset \calN_{\kappa}(b, \nn)
\qquad\text{and}\qquad 
b \subset \calN_{\kappa}(\beta, \nn).
\]
Let $x, y$ be points in $X$ so that $d_X(x,y) \leq d_X(x, \beta)$. We need to find an 
upper-bound for $d_X(x', y')$, where $x' \in \pi_{\beta} (x), y' \in \pi_{\beta} (y)$. For the remainder of the proof, we use $x_{\beta}$ to denote a point in the set $\pi_{\beta} (x)$ and $y_{\beta}$ to denote a point in $\pi_{\beta}(y)$. The upper-bound certainly exists if $x \in \calN_{\kappa}(\beta, \nn)$. Thus assume $d(x, \beta) \geq \nn \kappa(x)$.

We claim that there is a point $z$ along the 
geodesic segment $[x, x_\beta]$ such that 
\[
d_X(x, z) \leq d_X(x, b)
\qquad\text{and}\qquad 
d_X(z, b) \leq 3 \nn \cdot \kappa(x).
\]
To see this, note that 
\begin{equation} \label{Eq:b-beta}
d_X(x, \beta) \leq d_X(x, x_b) + d_X(x_b, \beta) 
  \leq d_X(x, x_b) + \nn \cdot \kappa(x_b).
\end{equation}
Meanwhile, the projection of the segment $[\go, x]$ to the geodesic $b$ is the segment $[\go, x_{b}]$. Since projections in \CAT spaces are Lipschitz, $\Norm{x_b} \leq \Norm{x}$. Thus $\kappa( x_{b} ) \leq \kappa (x).$ Therefore, if we choose $z$ to have 
distance $\nn \cdot \kappa(x)$ from $x_\beta$, we are sure to have
$d_X(z,x) \leq d_X(x, b)$.  Also, 
\begin{equation} \label{Eq:b}
d_X(z, b) \leq d_X(z, z_\beta) + d_X(z_\beta, b) 
  \leq \nn \cdot \kappa(x) + \nn \cdot \kappa(x_\beta). 
\end{equation}
Now, note that
\[
\Norm{x_\beta} \leq \Norm{x} + d_X(x, x_\beta) \leq 2 \Norm{x}. 
\] 
Hence, $\kappa({x_\beta}) \leq 2 \kappa(x)$. This and \eqnref{Eq:b}
imply the second assertion in the claim. 

Now, since $b$ is contracting, 
\[
d_X(z_b, x_b) \leq \cc_b \cdot \kappa(x). 
\]
Therefore, 
\begin{align}
d_X(x_b, x_\beta) 
 &\leq d_X(x_b, z) + d_X(z, x_\beta) \notag \\
 &\leq 3\nn \cdot \kappa(x) + \nn \cdot \kappa(x)
    =4 \nn \cdot \kappa(x). \label{Eq:4n}
\end{align}
                      
Now let $x, y \in X$ be such that $d_X(x, y) \leq d_X(x, \beta)$. Note that,  
\[
\Norm{y} \leq \Norm{x} + d_X(x,y)
 \leq \Norm{x} + d_X(x,\beta)  \leq 2 \Norm x. 
\]
Hence, applying \eqnref{Eq:4n}
to $x$ and $y$ we have 
\[
d_X(x_b, x_\beta) \leq 4 \nn \cdot \kappa(x)
\qquad\text{and}\qquad 
d_X(y_b, y_\beta) \leq 4 \nn \cdot \kappa(y) \leq 8 \nn \cdot \kappa(x). 
\]
Also, from \eqnref{Eq:b-beta}, we have 
\[
d_X(x, b) \geq d_X(x, \beta) - \nn \cdot \kappa(x_b)
  \geq d_X(x, y) - \nn \cdot \kappa(x). 
\]
Therefore, there is a point $y' \in [x,y]$ with 
\[
d_X(x, y') \leq d_X(x, b) 
\qquad\text{and}\qquad 
d(y, y') \leq \nn \cdot \kappa(x). 
\]
Thus, since closest-point projection is distance non-increasing,
\begin{align*}
d_X(x_\beta, y_\beta)
 & \leq d_X(x_b, y_b) + d_X(x_b, x_\beta) + d_X(y_b, y_\beta)\\
 & \leq d_X(x_b, y'_b) + d_X(y'_b, y_b)
     + 12 \cdot \nn \cdot \kappa( x)\\
 & \leq \cc_b \cdot \kappa(x) + \nn \cdot \kappa( x) + 
 12 \cdot \nn \cdot \kappa( x)\\
 & \leq (\cc_b + 13 \nn) \cdot\kappa( x). 
\end{align*} 
That is $\beta$ is $\kappa$--contracting with $\cc_\beta = (\cc_b + 13 \nn)$.
\end{proof}

We now prove that every $\kappa$-contracting set is also strongly $\kappa$-Morse. This
in particular proves the implication $(2) \Longrightarrow (3)$.

\begin{theorem}[Contracting implies strongly $\kappa$--Morse] \label{Thm:Strong}
Let $Z$ be a closed subspace that is $\kappa$--contracting. Then $Z$ is 
strongly $\kappa$-Morse. 
\end{theorem} 

\begin{proof}
Let $\cc_Z$ be the contracting constants for $Z$. Set 
\begin{equation} \label{Eq:Conditions}
\mm_0 = \qq\big( (\qq+1) + \qq \cc_Z + \sQ\big)
\qquad\text{and}\qquad
\mm_1= q\cc_Z +\qq + \sQ. 
\end{equation} 

\begin{claim*}
Consider a time interval $[s,s']$ during which $\eta$ is outside of 
$\calN_{\kappa}(Z, \mm_0)$. Then 
\begin{equation}  \label{Eq:End-Point} 
|s'-s| \leq \mm_1 \big( d_X\big(\eta(s), Z\big)+ d_X\big(\eta(s'), Z\big) \big). 
\end{equation} 
\end{claim*} 

\begin{proof}[Proof of the Claim] \renewcommand{\qedsymbol}{$\blacksquare$}
Let 
\[
s = t_{0} < t_{1} < t_{2}< \dots < t_{\ell} = s'
\]
be a sequence of times such that, for $i=0, \dots, {\ell-2}$, we have $t_{i+1}$ is a first time after $t_i$ where 
\[
d_X\big(\eta(t_i), \eta(t_{i+1}) \big) = d_X(\eta(t_i), Z)
\quad\text{and}\quad 
d_X\big(\eta(t_{\ell-1}), \eta(t_\ell) \big) \leq d_X(\eta(t_{\ell-1}), Z).
\]
To simplify the notation, we define
\[
\eta_i = \eta(t_i), \qquad \rr_i= \Norm{\eta(t_i)}, \qquad
\dd_i= d_X(\eta_i, Z)
\qquad\text{and}\qquad
\pi_i = (\eta_i)_Z.
\]
Recall that $(\eta_i)_Z$ is the set of the closest points in $Z$ to $\eta_i$. 
Note that, by assumption
\[
\dd_i \geq \mm_0 \cdot \kappa(\rr_i ).
\]

Since $Z$ is contracting, 
\[
d_X\big( \pi_0 , \pi_\ell \big) \leq 
\sum_{i=0}^{\ell-1} \diam_X \big( \pi_i  , \pi_{i+1}  \big) 
\leq \sum_{i=0}^{\ell-1} \cc_Z \cdot \kappa(\rr_i). 
\]
But $\eta$ is a $(\qq, \sQ)$--quasi-geodesic, hence, 
\begin{align} 
|s'-s| & \leq \qq \, d_X(\eta_0, \eta_\ell) + \sQ \notag \\
  & \leq \qq \left( \dd_0 + d_X\big( \pi_0 , \pi_\ell \big)  + \dd_\ell \right) + \sQ 
  \label{Eq:Upper} \\
  & \leq  \qq \, \cc_Z  \left( \sum_{i=1}^{\ell-1}  \kappa(\rr_i) \right) +  
    \qq \, (\dd_0 + \dd_\ell) + \sQ.  \notag
\end{align}
On the other hand, 
\[
|s'-s| = \sum_{i=0}^{\ell-1} |t_{i+1}- t_i| \ge 
\sum_{i=0}^{\ell-1} \left( \frac 1\qq d_X(\eta_i, \eta_{i+1}) -\sQ \right). 
\]
But, for $i=0, \dots, (\ell-2)$ we have $d_X(\eta_i, \eta_{i+1}) = \dd_i$ and 
\[
d_X(\eta_{\ell-1}, \eta_\ell) \geq \dd_{\ell-1} - \dd_\ell.
\]
Hence, 
\begin{align}\label{Eq:Lower}
|s'-s| & \geq 
  \sum_{i=0}^{\ell-1} \left( \frac{\mm_0}\qq \cdot \kappa(\rr_i) -\sQ \right) 
   -  \frac{\dd_\ell}\qq.
\end{align}
Combining \eqnref{Eq:Upper} and \eqnref{Eq:Lower} we get 
\[
\qq \, (\dd_0 + \dd_\ell) + \sQ + \frac{\dd_\ell}\qq \geq  
  \left( \frac{\mm_0}{\qq} - \qq \, \cc_Z - \sQ \right)  \sum_{i=0}^{\ell-1} \kappa(\rr_i).
\]
But, from \eqref{Eq:Conditions}, we have $\sQ \leq \mm_0 \leq \rr_0$ and 
\[
\qq \, (\dd_0 + \dd_\ell) + \sQ + \frac{\dd_\ell}\qq  \leq (\qq + 1)(\dd_0 + \dd_\ell) 
\qquad\text{and}\quad
\left(\frac{\mm_0}{\qq} - \qq \, \cc_Z - \sQ\right)  \geq  (\qq+1). 
\]
Which implies
\begin{equation*}
\sum_{i=0}^{\ell-1} \kappa(\rr_i) \leq \dd_0 + \dd_\ell
\qquad\text{and by \eqnref{Eq:Upper}}\qquad
|s'-s| \leq \mm_1 (\dd_0 + \dd_\ell). 
\end{equation*}
This proves the claim. 
\end{proof} 

\begin{figure}[h!]
\begin{tikzpicture}[scale=0.9]
 \tikzstyle{vertex} =[circle,draw,fill=black,thick, inner sep=0pt,minimum size=.5 mm]
[thick, 
    scale=1,
    vertex/.style={circle,draw,fill=black,thick,
                   inner sep=0pt,minimum size= .5 mm},
                  
      trans/.style={thick,->, shorten >=6pt,shorten <=6pt,>=stealth},
   ]

  \node[vertex] (o) at (0,0)[label=left:$\go$] {}; 
  \node(a) at (11,0)[label=right:$Z$] {}; 
  
  \draw (o)--(a){};
  \draw [dashed] (0, 0.4) to [bend left = 10] (10,2){};
  \node at (8.5,2.3){$\mm_{0} \cdot \kappa(\sR)$};
  \draw [dashed] (0, 1.4) to [bend left = 10] (5,3){};
  \node at (3.2,3.2){$\mm_{Z}(\qq, \sQ) \cdot \kappa(\rr)$};
     
 \draw [dashed] (10, 4.5) to (10,0) {};
 \node at (10,0)[label=below:$\sR$] {};
 \draw [dashed] (5, 4.5) to (5,0){};
 \node at (5,0)[label=below:$\rr$] {};
        
  \draw [decorate,decoration={brace,amplitude=10pt},xshift=0pt,yshift=0pt]
  (10,3.3) -- (10,0)  node [thick, black,midway,xshift=0pt,yshift=0pt] {};       
 \node at (11.3,1.7) {$\nn \cdot \kappa'(\sR)$};
        
 \node[vertex] at (5.47,1.8)[label=below right:$t_{\rm last}$] {}; 
 \node[vertex] at (2.04, 1.05)[label=below right:$s$] {}; 
 \node[vertex] at (4.34,1.58)[label=above right:$s'$] {}; 
 
  \pgfsetlinewidth{1pt}
  \pgfsetplottension{.75}
  \pgfplothandlercurveto
  \pgfplotstreamstart
  \pgfplotstreampoint{\pgfpoint{0cm}{0cm}}  
  \pgfplotstreampoint{\pgfpoint{1cm}{-.6cm}}   
  \pgfplotstreampoint{\pgfpoint{2cm}{1cm}}
  \pgfplotstreampoint{\pgfpoint{3cm}{2.4cm}}
  \pgfplotstreampoint{\pgfpoint{4cm}{2cm}}
  \pgfplotstreampoint{\pgfpoint{5cm}{1cm}}
  \pgfplotstreampoint{\pgfpoint{6cm}{3cm}}
  \pgfplotstreampoint{\pgfpoint{7cm}{3.2cm}}
  \pgfplotstreampoint{\pgfpoint{8cm}{3.6cm}}
  \pgfplotstreampoint{\pgfpoint{9cm}{3.1cm}}
  \pgfplotstreampoint{\pgfpoint{10cm}{3.3cm}}
  \pgfplotstreamend
  \pgfusepath{stroke} 
  \end{tikzpicture}
  
\caption{The concatenation of a geodesic segment $[x,y]$ and the 
quasi-geodesic segment $[y,z_1]$ is a quasi-geodesic.}
\label{Fig:Strong} 
\end{figure}
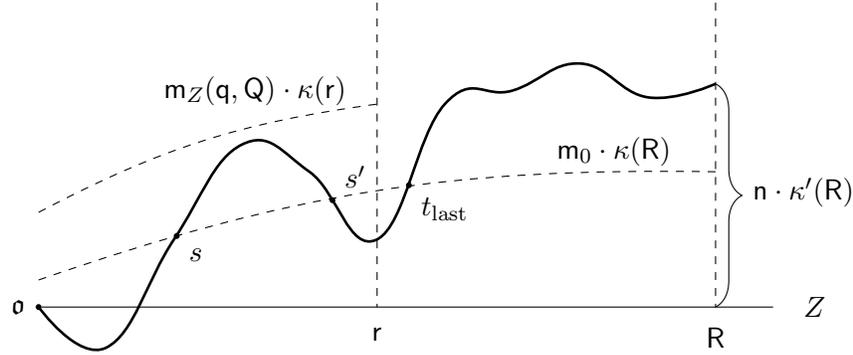

Now let $t_{\rm last}$ be the last time $\eta$ is in $\calN_{\kappa}(Z, \mm_0)$ and 
consider the quasi-geodesic path $\eta[t_{\rm last}, t_\sR]$. Since this path is 
outside of $\calN_\kappa(Z, \mm_0)$, we can use \eqnref{Eq:End-Point} to get 
\[
|\sR-t_{\rm last} | \leq \mm_1 \big( d_X(\eta(t_{\rm last}), Z)+ d_X(\eta(t_\sR), Z) \big). 
\]
But 
\[
d_X(\eta(t_{\rm last}), Z) \leq \mm_0 \cdot \kappa(\eta(t_{\rm last})) 
  \leq \mm_0 \cdot \kappa(\sR)
  \qquad\text{and}\qquad
 d_X(\eta(t_\sR), Z) \leq  \nn \cdot \kappa'(\sR).
\]
Therefore,
\[
|\sR-t_{\rm last} | \leq \mm_0 \cdot \mm_1\cdot \kappa(\sR)+ \nn \cdot \kappa'(\sR)
\]
Since $\mm_0$, $\mm_1$ and $\nn$ are given and $\kappa$ and $\kappa'$ are 
sublinear, there is a value of $\sR$ depending on $\mm_0$, $\mm_1$, $\nn$, $\rr$, 
$\kappa$ and $\kappa'$ such that 
\begin{equation} \label{Eq:R-defined}
 \mm_0 \cdot \mm_1 \cdot \kappa(\sR)+ \nn \cdot \kappa'(\sR) \leq \sR - \rr.
\end{equation}
For any such $\sR$, we then have 
\[t_{\rm last} \geq \rr.\] 

We show that $\eta[0, t_{\rm last}]$ stays in a larger $\kappa$--neighborhood
of $Z$. Consider any other subinterval $[s,s'] \subset [0, t_{\rm last}]$ where 
$\eta$ exits $\calN_\kappa(Z, \mm_0)$. By taking $[s,s']$ as large as possible, 
we can assume $\eta(s), \eta(s') \in \calN_\kappa(Z, \mm_0)$. 
In this case, 
\[
d_X(\eta(s), Z) \leq \mm_0 \cdot \kappa(\eta(s)) 
\qquad\text{and}\qquad
d_X(\eta(s'), Z) \leq \mm_0 \cdot \kappa(\eta(s')).  
\]
again applying \eqnref{Eq:End-Point}, we get 
\[
|s'-s| \leq \mm_0 \, \mm_1 \cdot \big(\kappa(\eta(s)) + \kappa(\eta(s'))\big).
\]
and thus 
\begin{align*}
\Big| \Norm{\eta(s')} - \Norm{\eta(s)} \Big|
 &\leq \qq \, \mm_0 \, \mm_1 \cdot \big(\kappa(\eta(s)) + \kappa(\eta(s'))\big) + \sQ\\
 &\leq (\qq \, \mm_0 \, \mm_1 +\sQ)\cdot \big(\kappa(\eta(s)) + \kappa(\eta(s'))\big)\\
 &\leq 2 (\qq \, \mm_0 \, \mm_1 +\sQ)\cdot \max \big(\kappa(\eta(s)) ,\kappa(\eta(s'))\big).
\end{align*}
Applying \lemref{Lem:sublinear-estimate}, we have that 
\[
\kappa(\eta(s')) \leq \mm_2 \cdot \kappa(\eta(s)),
\]
for some $\mm_2$ depending on $\cc_Z$, $\qq$, $\sQ$ and $\kappa$. 
Therefore, for any $t \in [s, s']$
\begin{equation} \label{Eq:t-s}
|t-s| \leq \mm_0 \, \mm_1 (1+ \mm_2) \cdot \kappa(\eta(s)).
\end{equation}
As before, this implies, 
\[
\Big| \Norm{\eta(t)} - \Norm{\eta(s)} \Big|
 \leq \qq \, \mm_0 \, \mm_1 (1+ \mm_2) \cdot \kappa(\eta(s)) + \sQ
 \leq (\qq \, \mm_0 \, \mm_1 (1+ \mm_2) +\sQ)\cdot \kappa(\eta(s)).
\]
Applying \lemref{Lem:sublinear-estimate} again, we have 
\begin{equation} \label{Eq:m3}
\kappa(\eta(s)) \leq \mm_3 \cdot \kappa(\eta(t)),
\end{equation}
for some $\mm_3$ depending on $\cc_Z$, $\qq$, $\sQ$ and $\kappa$. 

Now, for any $t \in [s, s']$ we have
\begin{align*}
d_X(\eta(t), Z)& \leq d_X(\eta(t), \eta(s)) + \rr_0 \\
& \leq \qq |t-s| + \sQ + \mm_0 \cdot \kappa(\eta(s))\\
& \leq (\qq \mm_0\, \mm_1 (1+\mm_2) + \sQ + \mm_0) \cdot \kappa(\eta(s))
  \tag{\eqnref{Eq:t-s}}\\
& \leq (\qq \mm_0\, \mm_1 (1+\mm_2) + \sQ + \mm_0) \, \mm_3 \cdot \kappa(\eta(t)).  
  \tag{\eqnref{Eq:m3}}
\end{align*}
Now setting 
\begin{equation}\label{mz}
\mm_Z(\qq, \sQ) = (\qq \mm_0\, \mm_1 (1+\mm_2) + \sQ + \mm_0) \, \mm_3
\end{equation}
we have that
\[
\eta[s, s'] \subset \calN_{\kappa}\big(Z, \mm_Z(\qq, \sQ)\big)
\qquad\text{and hence}\qquad 
\eta[0, t_{\rm last}] \subset \calN_{\kappa}\big(Z, \mm_Z(\qq, \sQ)\big).
\]
The $\sR$ we have chosen depends on the value of $\qq$ and $\sQ$. However, 
the assumption that $\mm_Z(\qq, \sQ)$ is small compared to $\rr$ (see \eqnref{Eq:Small}) 
gives an upper-bound for the values of $\qq$ and $\sQ$. Hence, we can choose $\sR$ to be the 
radius associated to the largest possible value 
for $\qq$ and the largest possible value for $\sQ$. This finishes the proof. 

Note that, the assumption that $\mm_Z(\qq, \sQ)$ is small compared to $\rr$ is not 
really needed here and any upper-bound on the values of $\qq$ and $\sQ$ would
suffice. But this is the assumption we will have later on and hence it is natural to state the
theorem this way. 
\end{proof}

\begin{remark}
As can be seen in \eqnref{Eq:R-defined}, the value of $\sR$ can be calculated explicitly, namely, 
$\sR$ depends on $\kappa$, $\kappa'$, $\qq$, $\sQ$, $\cc_Z$ and $\nn$. That is, all we need to 
know from the set $Z$ is the function $\kappa$ and the value of the constant $\cc_Z$. 
\end{remark}

We now show when $Z$ is the image of a quasi-geodesics ray, the notion of strongly $\kappa$-Morse 
is indeed stronger than the notion of weakly $\kappa$-Morse hence proving $(3) \Longrightarrow (5)$. 

\begin{lemma}[Strongly $\kappa$--Morse implies weakly $\kappa$-Morse] \label{Lem:Strong-Weak}
Let $Z$ be the image of a $(\qq_0, \sQ_0)$--quasi-geodesic ray $\beta$. If $Z$ is strongly $\kappa$--Morse then 
$Z$ is weakly $\kappa$--Morse.
\end{lemma} 

\begin{figure}[h!]
\begin{tikzpicture}
 \tikzstyle{vertex} =[circle,draw,fill=black,thick, inner sep=0pt,minimum size=.5 mm]
 
[thick, 
    scale=1,
    vertex/.style={circle,draw,fill=black,thick,
                   inner sep=0pt,minimum size= .5 mm},
                  
      trans/.style={thick,->, shorten >=6pt,shorten <=6pt,>=stealth},
   ]
   
    \node[vertex] (a) at (0,0) [label=$\go$]  {};
     \node[vertex] (b) at (1.13, 0.43) [label=above left:$\go_{\gamma}$]  {};
     \node[vertex]  at (3,0) [label=above left: $\gamma(s)$]  {};
     \node[vertex]  at (3,1.5) [label=$\gamma(u_{s})$]  {};
      \draw [dashed](3,1.5)--(3,0);
        \node[vertex]  at (6,0) [label=above right: $\gamma(t)$]  {};
        \node[vertex]  at (6,1.5) [label=$\gamma(u_{t})$]  {};
         \draw [dashed](6,1.5)--(6,0);
    \node (b1) at (13, 0) {};
       \node [vertex](c) at (12, 0) [label=$x$]   {};
        \node[vertex] (d) at (7.88, 0.43) [label=above right: $x_{\gamma}$]  {};
 \draw [dashed](c)--(d);
  \draw [thick ](a)--(b1);
    \draw [dashed](a)--(b);

    \draw[red, thick] (3,0) arc (330:90:1);
       \draw[red, thick] (6,0) arc (210:450:1);
    \draw [red, thick] (2,1.5) -- (7, 1.5);
       
    \node at (13.2, 0) {$\beta$};
    \node at (5, 1.2) {$\gamma_{1}$};    
     \node at (2, -0.7) {$\gamma_{0}$};    
      \node at (7, -0.7) {$\gamma_{2}$};    
\end{tikzpicture}
\caption{The quasi-geodesic segment $\gamma = \gamma_{0}\cup \gamma_{1} \cup \gamma_{2}$ is in a sublinear neighbourhood of $\beta$.}
\label{Fig:YoYo}
\end{figure}

\begin{proof}
Let $\gamma\from [s,t] \to X$ be a $(\qq, \sQ)$--quasi-geodesic with end points in $Z$. 
Assume $\gamma(s) = \beta(s')$ and $\gamma(t) = \beta(t')$. Let $\rr = \max \norm{\gamma(u)}$ for $u \in [s,t]$
and let $x \in Z$ be a point such that $\sR=\norm{x}$ is much larger than $\rr$ (to be determined later). 
Consider the points $x_\gamma$ (a point in the projection of $x$ to $\gamma$) and $\go_\gamma$
(a point in the projection of $\go$ to $\gamma$) and write $\gamma$ as a concatenation of 
\[
\gamma_0 = [\gamma(s), \go_\gamma]_\gamma, \qquad 
\gamma_1=[\go_\gamma, x_\gamma]_\gamma \qquad\text{and} \qquad
\gamma_2 = [x_\gamma, \gamma(t)]_\gamma.
\]
Since the projection map is coarsely Lipschitz, the shadow of $\gamma_1$  to $\beta$ coarsely 
covers $[\beta(s'), \beta(t')]_\beta$. That is, there is a constant $\sL >0$ depending only on $\sQ_0$ and $\qq_0$ 
and points $\gamma(u_s)$ and $\gamma(u_t)$ along $\gamma_1$ such that $\gamma(u_s)$ projects 
$\sL$--close to $\beta(s')$ and $\gamma(u_t)$ projects $\sL$--close to $\beta(t')$. 

Now, applying \lemref{Lem:surgery} twice, we have that the path 
\[
\gamma'=[\go, \go_\gamma]\cup [\go_\gamma, x_\gamma]_\gamma\cup[x_\gamma, x]
\]
is a $(81 \qq, \sQ)$--quasi-geodesic. For $\sR$ large enough, the condition of $\kappa$-strongly Morse implies 
that $\gamma'$ is contained in the $\kappa$--neighborhood $\calN_\kappa(Z, \mm_\beta(81 \qq, \sQ))$ 
of $\beta$. In particular, 
\[
d_X(\beta(s'), \gamma(u_s)) \leq  \mm_\beta(81 \qq, \sQ) \cdot \kappa( \gamma(u_s))
\quad\text{and}\quad  
d_X(\beta(s'), \gamma(u_s)) \leq \mm_\beta(81 \qq, \sQ) \cdot \kappa( \gamma(u_t)). 
\]
But $\gamma_0$ and $\gamma_2$ are $(\qq, \sQ)$--quasi-geodesics. Therefore, 
there is a constant $\sD$ depending on $\beta$, $\qq$ and $\sQ$ such that 
\[
|u_s-s| \leq \sD \cdot \kappa(\gamma(u_s)) 
\qquad\text{and}\qquad 
|t-u_t| \leq \sD \cdot \kappa(\gamma(u_s)). 
\]
Thus, $\gamma_0$ and $\gamma_1$ are not too long and they are entirely contained in 
a $\kappa$--neighborhood of $\beta$. And we have already shown that $\gamma_1$ which is a subsegment of 
$\gamma'$ is contained in a $\kappa$--neighborhood of $\beta$. Therefore, 
$\gamma$ itself is contained in a $\kappa$ neighborhood of $\beta$. 
\end{proof}

This concludes the proof of \thmref{Thm:TFAE}. We finish with a couple of corollaries of \thmref{Thm:Strong}. 
Recall that, 
a $(\qq, \sQ)$--quasi-geodesic $\beta$ is in $\bfb$ if $\beta$ 
is contained in some $(\kappa, \nn)$--neighborhood of the geodesic ray $b\in \bfb$. 
A priori, it might be possible for the constant $\nn$ to go to infinity even 
as $\qq$ and $\sQ$ remain bounded. However, this does not happen. 

\begin{corollary} \label{Cor:m_b}
Let $b$ be a $\kappa$--contracting geodesic ray and let $\mm_b$ be as in 
\thmref{Thm:Strong} (where $Z$ is the image of $b$). Then, for any 
$(\qq, \sQ)$--quasi-geodesic $\beta \in [b]$, we have 
\[
\beta \subset \calN_{\kappa} \big(b, \mm_b(\qq, \sQ)\big)
\qquad\text{and}\qquad
b \subset \calN_{\kappa} \big(\beta, 2 \mm_b(\qq, \sQ)\big).
\]
\end{corollary}

\begin{proof} 
Since $\beta \in [b]$, there is a constant $\nn$ so that $\beta \subset \calN_{\kappa}(b, \nn)$. 
For every $\rr$, let $t_\rr$ be the first time when $\beta(t_\rr)$ has norm $\rr$. 
We have 
\[
d_X(\beta(t_\sR), b) \leq \nn\cdot \kappa(\sR)
\] 
for every $\sR$. Now \thmref{Thm:Strong} implies that 
\[
\beta[0, t_\rr] \subset  \calN_{\kappa}\big(b, \mm_b(\qq, \sQ)\big)
\]
for every $\rr$. This proves the first assertion. 

To see the second assertion, consider a point $b_\rr = b(\rr)$, let 
$\beta_\rr=\beta(t_\rr)$ and let $q= \pi_b(\beta_\rr)$. Then,
the first assertion implies 
\[
d_X(\beta_\rr, q) \leq \mm_b(\qq, \sQ) \cdot \kappa(\rr). 
\]
Hence, 
\begin{align*}
d_X(b_\rr, q) & \leq \rr - d_X(\go, q)\\
  & \leq \rr - \big(d_X(\go, \beta_\rr) - d_X(\beta_\rr, q) \big) 
      \leq \mm_b(\qq, \sQ) \cdot \kappa(\rr). 
\end{align*}
Therefore, 
\[
d_X(b_\rr,  \beta) \leq d_X(b_\rr, \beta_\rr) \leq d_X(b_\rr,q) + d_X(q, \beta_\rr)
\leq 2\mm_b(\qq, \sQ) \cdot \kappa(\rr),
\]
which implies $b \subset \calN_{\kappa}\big(\beta, 2 \mm_b(\qq, \sQ)\big)$. 
\end{proof} 

\begin{corollary} \label{Cor:m_beta}
If $\beta \in \bfb$ is a $(\qq, \sQ)$--quasi-geodesic, then the function
\[
\mm_\beta(\param, \param) \leq \mm_b(\param, \param) + 2 \mm_b(\qq, \sQ)
\]
is a Morse gauge for $\beta$. In particular, the Morse gauge depends only on
$\mm_b$, $\qq$ and $\sQ$ and not on the particular quasi-geodesic $\beta$. 
\end{corollary}

\begin{proof}
Let $\beta' \in \bfb$ be a $(\qq', \sQ')$--quasi-geodesic. Let $\beta'_\rr$ be 
a point along $\beta'$ with norm $\rr$, let $p=\pi_b(\beta'_\rr)$ and let $q$ be the 
closest point in $\beta$ to $p$. Note that $\Norm p \leq \rr$. Hence, 
\begin{align*} 
d_X(\beta'_\rr, \beta) &\leq d_X(\beta'_\rr, p) + d_X(p, q) \\
  &\leq \mm_b(\qq', \sQ') \cdot \kappa(\rr) + 2 \mm_b(\qq, \sQ) \cdot \kappa(p)
  \leq \big( \mm_b(\qq', \sQ')+ 2 \mm_b(\qq, \sQ)\big) \cdot \kappa(\rr)
\end{align*} 
This finishes the proof. 
\end{proof}

\section{$\kappa$--Morse boundary} \label{Sec:Topology} 

Recall the definition of $\kappa$--fellow traveling (\defref{Def:Fellow-Travel}) which 
defines an equivalence relation on the set of all quasi-geodesic rays in $X$. Recall 
also that all geodesic rays and quasi-geodesic rays are assumed to start from the fixed 
base-point $\go$. 

\begin{definition}[$\kappa$--Morse boundary set] \label{Def:Boundary}
The $\kappa$--Morse boundary of $X$, $\partial_\kappa X$, is the set of 
$\kappa$--equivalence classes quasi-geodesic rays in $X$ that satisfy any one
of the equivalent properties given in \thmref{Thm:TFAE}. 
Since each class contains a unique geodesic which is $\kappa$--contracting 
(again, by \thmref{Thm:TFAE}) we could also define $\partial_\kappa X$ to be 
the set of $\kappa$--contracting geodesic rays in $X$. 
\end{definition}

We equip $\partial_\kappa X$ with a topology which is 
a coarse version of the visual topology. Roughly speaking, we think of a point 
$\bfa \in \partial_\kappa X$ as being in a small neighborhood of $\bfb \in \partial_\kappa X$ 
if, for some large radius $\rr$, every $(\qq, \sQ)$--quasi-geodesic $\alpha \in \bfa$,
where $\mm_b(\qq,\sQ)$ is small compared to the radius $\rr$, 
fellow travels the geodesic $b \in \bfb$ up to the radius $\rr$. 
As we shall see, this is strictly stronger than assuming that the geodesics $a \in \bfa$ 
and $b \in \bfb$ fellow travel each other for a long time. 

We introduce the following notations. Let $\beta$ be a $(\qq, \sQ)$--quasi-geodesic ray that is $\kappa$--Morse and let 
$\mm_\beta$ be the associated Morse gauge functions as in \thmref{Thm:Strong}. 
For $\rr>0$, let $t_\rr$ be the first time where $\Norm{\beta(t)}=\rr$ and define:
\[
\beta_\rr = \beta(t_\rr)
\qquad\text{and}\qquad
\beta|_{\rr} = \beta{[0,t_\rr]}
\] 

which we consider as a subset of $X$.
\begin{definition}[neighbourhoods]\label{neighbourhoods}
 Let $\bfb \in \partial_\kappa X$ and $b \in \bfb$ 
be the unique geodesic in the class $\bfb$. Define $\calU_\kappa(\bfb, \rr)$ to be the set 
of points $\bfa \in \partial_\kappa X$ such that, for any $(\qq, \sQ)$--quasi-geodesic 
$\alpha \in \bfa$ where $\mm_b(\qq, \sQ)$ is small compared to $\rr$ 
(see \eqnref{Eq:Small}) we have 
\[
\alpha|_{\rr} \subset \calN_{\kappa}\big(b, \mm_b(\qq, \sQ)\big). 
\]

\end{definition}

\subsection{Neighborhood system} 
In this sub-section we show that the sets $\calU_{\kappa}(\gamma, \rr)$ generate a neighborhood 
system which can be used to define a topology for $\pka X$.
We start with a technical lemma.

\begin{lemma}\label{Lem:surgery}
Let  $X$ be a proper, complete metric space. Let $b$ be a geodesic ray and $\gamma$ be a $(\qq, \sQ)$--quasi-geodesic ray. 
For $\rr>0$, assume that $d_X(b_\rr, \gamma)\leq \rr/2$. Then, there exists a
$(9\qq,\sQ)$--quasi-geodesic $\gamma'$ so that 
\[
\gamma' \in [b], \qquad\text{and}\qquad \gamma|_{\rr/2} = \gamma'|_{\rr/2}. 
\]
\end{lemma}
\begin{proof}
Let $q$ be a point in $\gamma$ that is closest to $b_{\rr}$ and let $\sR > 0$ be such that  the ball of radius $\sR$ centered at $\go$ contains $[\go, q]_{\gamma}$. Now let $q' $ be the point in $[0, q]_{\gamma}$ closest to $b_{\sR}$. Then
\begin{align*}
\Norm {q'} &\geq \Norm{b_\sR} - d_{X}(b_{\sR}, q')\\   
               & \geq \sR - d_{X}(b_{\sR}, q)\\   
               & \geq \sR - \Big( d_{X}(b_{\sR}, b_{\rr}) + d_{X}(b_{\rr}, q)\Big)\\ 
               & \geq \sR - (\sR - \rr) - \frac \rr2 = \frac \rr 2
 \end{align*}

Applying \lemref{concatenation}, we have a 
$(3\qq, \sQ)$--quasi-geodesic 
segment
\[
\zeta = [\go, q']_\gamma \cup [q', b_{\sR}].
\]
Furthermore, by construction, $\Norm {q'} \leq \sR = \Norm {b_{\sR}}$. Therefore, the 
projection of any point on the geodesic $b[\sR, \infty)$ to $\zeta$ is the point 
$b_{\sR}$. Applying \lemref{concatenation} again we have that the concatenation 
\[
\gamma'=\zeta \cup b[\sR, \infty)
\]
is a $(9\qq, \sQ)$--quasi-geodesic ray.
\begin{figure}[H]
\begin{tikzpicture}[scale=0.6]
 \tikzstyle{vertex} =[circle,draw,fill=black,thick, inner sep=0pt,minimum size=.5 mm]
 
[thick, 
    scale=1,
    vertex/.style={circle,draw,fill=black,thick,
                   inner sep=0pt,minimum size= .5 mm},
                  
      trans/.style={thick,->, shorten >=6pt,shorten <=6pt,>=stealth},
   ]

  \node[vertex] (o) at (-5,0)  [label=left:$\go$] {}; 
  \node[vertex] (q) at (1.3,0.15)  [red, label=above:$q$] {};
   \node[vertex] (q1) at (2,-0.2)  [red, label=right:$q'$] {};
  \node (xi) at (5.3,1) [label=below:$\gamma$] {}; 
  \node (g) at (5.3,2.3) [label=below:$b$] {}; 

    \node [vertex] (p) at (1-1,0.85)  [label=above right:$b_{\rr}$] {};
     \node [vertex] (p1) at (2+0.8,1.3)  [label=above right:$b_{\sR}$] {};
    \draw[thick]  (o)--(5,1.66){};

     \draw [dashed](p)--(q){};
        \draw[dashed] (p1)--(q1){};
     \draw [-, dashed] (1-1.3, -2) to [bend right = 20] (1-1.3, 2){};
     \draw [-, dashed] (2+0.6, -2) to [bend right = 20] (2+0.6, 2){};
     \node at (0.5-1, -2) [label=right:$\rr$] {};
          \node at (2.4, -2) [label=right:$\sR$] {};

  \pgfsetlinewidth{1pt}
  \pgfsetplottension{.75}
  \pgfplothandlercurveto
  \pgfplotstreamstart
  \pgfplotstreampoint{\pgfpoint{-5cm}{0cm}}  
  \pgfplotstreampoint{\pgfpoint{-4cm}{-.4cm}}   
  \pgfplotstreampoint{\pgfpoint{-3cm}{-0.4cm}}
  \pgfplotstreampoint{\pgfpoint{-2cm}{-0.1cm}}
  \pgfplotstreampoint{\pgfpoint{-1cm}{-.35cm}}
  \pgfplotstreampoint{\pgfpoint{0cm}{-0.5cm}}
  \pgfplotstreampoint{\pgfpoint{1cm}{-.9cm}}
  \pgfplotstreampoint{\pgfpoint{2cm}{-0.2cm}}
   \pgfplotstreampoint{\pgfpoint{1cm}{-0.2cm}}
      \pgfplotstreampoint{\pgfpoint{1.5cm}{0.2cm}}
  \pgfplotstreampoint{\pgfpoint{3cm}{.2cm}}
   \pgfplotstreampoint{\pgfpoint{4cm}{-.2cm}}
  \pgfplotstreampoint{\pgfpoint{5cm}{.4cm}}
  \pgfplotstreamend 
    \pgfusepath{stroke}
       
    \end{tikzpicture}
    \caption{The concatenation of $[\go, q']_\gamma$, $[q', b_\sR]$ and $b[\sR, \infty)$
    is a $(9\qq, \sQ)$--quasi-geodesic in the class $\bfb$.}
 \end{figure}
 
Lastly, since  $\Norm {q'} \geq \rr/2$, we have 
$\gamma|_{\rr/2} = \zeta|_{\rr/2} = \gamma'|_{\rr/2}$.
\end{proof}

\begin{proposition}\label{Prop:Normal}
For each $\bfb \in \partial_\kappa X$ and $\rr > 0$, there exists a radius 
$\rr_\bfb$ such that 
\begin{enumerate} 
\item for any point $\bfa$ there exists $\rr_\bfa$ so that 
\[
\bfa \in \calU_{\kappa}(\bfb, \rr_\bfb) 
\qquad\Longrightarrow\qquad
\calU_{\kappa}(\bfa, \rr_\bfa) \subset \calU_{\kappa}(\bfb, \rr).
\]
\item for any point $\bfa$ there exists $\rr_\bfa$ so that
\[
\bfa  \notin \calU_{\kappa}(\bfb, \rr) 
\qquad\Longrightarrow\qquad
\calU_{\kappa}(\bfa, \rr_\bfa) \cap \calU_{\kappa}(\bfb , \rr_\bfb) = \emptyset. 
\]
\end{enumerate} 
\end{proposition}

\begin{proof}
For the rest of this proof, we assume $\qq$ and $\sQ$ are such that if 
$\mm_a(\qq', \sQ')$ is small compared to $\rr$ then $\qq' \leq \qq$ and $\sQ' \leq \sQ$. 
Hence, if we prove a statement for all $(\qq, \sQ)$--quasi-geodesics, we have also 
shown the statement for all $(\qq', \sQ')$--quasi-geodesics where $\mm_a(\qq', \sQ')$ 
is small compared to $\rr$. 

Let $b \in \bfb$ be the unique geodesic ray in $\bfb$. We choose $\rr_\bfb$ such that
\[
\rr_\bfb \geq 2 \rr 
\qquad\text{and}\qquad
\mm_b(9\qq, \sQ) \leq \frac{\rr_\bfb}{2\kappa(\rr_\bfb)}. 
\]
Also, letting $\nn(\qq, \sQ) = \mm_b(9\qq, \sQ)$, we require that 
$\rr_\bfb \geq \sR$ where $\sR=\sR(b, \rr, \nn, \kappa)$ is as in \thmref{Thm:Strong}. 

\subsection*{Proof of Part (1)} Let  $\bfa \in \calU_{\kappa}(\bfb, \rr_\bfb)$ and let $a \in \bfa$ be 
the unique geodesic ray in $\bfa$. Choose $\rr_\bfa$ such that, 
\[
\rr_\bfa \geq 2 \rr_\bfb
\qquad\text{and}\qquad
\mm_a(\qq, \sQ) \leq \frac{\rr_\bfa}{4\kappa(\rr_\bfa)}.
\]
Now consider $\bfc \in \calU_{\kappa}(\bfa, \rr_\bfa)$ and let $\gamma \in \bfc$ be a
$(\qq, \sQ)$--quasi-geodesic. 
The proof of \corref{Cor:m_b} also shows that 
\[
d_X(a(\rr_\bfa), \gamma|_{\rr_\bfa}) \leq 2 \mm_a(\qq, \sQ) \cdot \kappa(\rr_\bfa)
\leq \frac{\rr_\bfa}2. 
\]
We apply \lemref{Lem:surgery}, with radius being $\rr_\bfa$, to modify $\gamma$ 
to a $(9\qq, \sQ)$--quasi-geodesic $\gamma' \in \bfa$. Since, $\rr_\bfb \leq \rr_\bfa/2$, 
we have $\gamma|_{\rr_\bfb} = \gamma'|_{\rr_\bfb}$. Also, 
$\bfa \in \calU_{\kappa}(\bfb, \rr_\bfb)$ and $\mm_b(9\qq, \sQ)$ is small 
compare to $\rr_\bfb$, therefore
\[
 \gamma|_{\rr_\bfb} =\gamma' |_{\rr_\bfb} \subset \calN_\kappa\big(a, \mm_a(9\qq, \sQ)\big). 
\]
But  $\gamma|_{\rr_\bfa} $ is actually a $(\qq, \sQ)$--quasi-geodesic. Hence, 
\thmref{Thm:Strong} (with  $\nn(\qq, \sQ) = \mm_a(9\qq, \sQ)$) implies that 
\[
\gamma|_\rr \subseteq \calN_\kappa\big(a, \mm_a(\qq, \sQ)\big). 
\]
This holds for every such $\gamma \in \bfc$, thus $\bfc \in \calU_{\kappa}(\bfb, \rr)$.
And this argument holds for every $\bfc \in \calU_{\kappa}(\bfa, \rr_\bfa)$, 
therefore $\calU_{\kappa}(\bfa, \rr_\bfa) \subset \calU_{\kappa}(\bfb, \rr)$.

\subsection*{Proof of Part (2)} In view of \corref{Cor:m_beta}, there exists a
constant $\uu>0$, depending on $\qq$ and $\sQ$, such that, for any $(\qq, \sQ)$--quasi-geodesic $\alpha \in \bfa$ 
we have
\[
 \mm_\alpha(1,0)+ 2\mm_a(\qq, \sQ) \leq \uu.
\]
Choose $\rr_\bfa$ large enough so that 
\[
\rr_\bfa \geq \max \big( 2\uu \cdot \kappa(\rr_\bfa) , 2 \rr_\bfb \big).
\]
Assume $\calU_{\kappa}(\bfa, \rr_\bfa) \cap \calU_{\kappa}(\bfb , \rr_\bfb)$ is non-empty and consider 
a point $\bfc$ in this set. Let $c \in \bfc$ be the unique geodesic ray in this class. 
We have to show $\bfa \in \calU_{\kappa}(\bfb, \rr)$. 

Consider a $(\qq, \sQ)$--quasi-geodesic $\alpha \in \bfa$. Since, 
$\bfc \in \calU_{\kappa}(\bfa, \rr_\bfa)$,
\[
d_X(c(\rr_\bfa), a) \leq \mm_a(1,0) \cdot \kappa(\rr_\bfa).
\] 
Defining $p= \pi_a(c(\rr_\bfa))$, we have $\Norm{p} \leq \rr_\bfa$. Therefore, 
the second assertion in \corref{Cor:m_b} implies 
\[
d_X(p, \alpha) \leq 2 \mm_a(\qq,\sQ) \cdot \kappa(p)
\leq 2 \mm_a(\qq,\sQ) \cdot \kappa(\rr_\bfa).
\]  
Hence,
\[
d_X(c(\rr_\bfa), \alpha) \leq
  d_X(c(\rr_\bfa), p) + d_X(p, \alpha) \leq 
  \uu \cdot \kappa(\rr_\bfa) \leq \frac{\rr_\bfa}2. 
\]
We can now apply \lemref{Lem:surgery} to $\alpha$ and $c$ with radius 
$\rr_\bfa$ to obtain a $(9\qq, \sQ)$--quasi-geodesic $\alpha' \in \bfc$
where (using $\frac{\rr_\bfa}{2} \geq \rr_\bfb$), $\alpha'|_{\rr_\bfb} = \alpha|_{\rr_\bfb}$. 

\begin{figure}[H]
\begin{tikzpicture}[scale=0.6]
 \tikzstyle{vertex} =[circle,draw,fill=black,thick, inner sep=0pt,minimum size=.5 mm]
 
[thick, 
    scale=1,
    vertex/.style={circle,draw,fill=black,thick,
                   inner sep=0pt,minimum size= .5 mm},
                  
      trans/.style={thick,->, shorten >=6pt,shorten <=6pt,>=stealth},
   ]

  \node[vertex] (o) at (-5,0)  [label=left:$\go$] {}; 
  \node (g) at (5.3,2.5) [label=center:$b$] {}; 
  \node (h) at (5.3,1.2) [label=center:$c$] {}; 
  \node (i) at (5.3,0) [label=center:$a$] {}; 
  \node (j) at (5.3,-1) [label=center:$\alpha$] {}; 

  \draw[thick]  (o)--(g){};
  \draw[thick]  (o)--(h){};
  \draw[thick]  (o)--(i){};

  \draw [-, dashed] (-3, -2) to [bend right = 20] (-3, 2){};
  \draw [-, dashed] (-1, -2) to [bend right = 20] (-1, 2){};
  \draw [-, dashed] (1, -2) to [bend right = 20] (1, 2){};
  \draw [-, dashed] (3, -2) to [bend right = 20] (3, 2){};

  \node at (-3, 2) [label=above:$\rr$] {};
  \node at (-1, 2) [label=above:$\rr_{\bfb}$] {};
  \node at (1, 2) [label=above:$\rr_{\bfa}/2$] {};
  \node at (3, 2) [label=above:$\rr_{\bfa}$] {};
  \node[vertex, red] (a) at (3.3, 1){};
  \node[vertex] (b) at (3.2, 0){};
  \draw [dashed](a)--(b){};
  
  \node at (3.1, 0) [label=below:$p$]{};  
  
  \draw [dashed](3.2, 0)--(2.5, -1){};
       
  \pgfsetlinewidth{1pt}
  \pgfsetplottension{.75}
  \pgfplothandlercurveto
  \pgfplotstreamstart
  \pgfplotstreampoint{\pgfpoint{-5cm}{0cm}}  
  \pgfplotstreampoint{\pgfpoint{-4cm}{-.7cm}}   
  \pgfplotstreampoint{\pgfpoint{-3cm}{-.5cm}}
  \pgfplotstreampoint{\pgfpoint{-2cm}{-.6cm}}
  \pgfplotstreampoint{\pgfpoint{-1cm}{-.55cm}}
  \pgfplotstreampoint{\pgfpoint{0cm}{-1cm}}
  \pgfplotstreampoint{\pgfpoint{1cm}{-1.9cm}}
  \pgfplotstreampoint{\pgfpoint{2.3cm}{-1.05cm}}
  \pgfplotstreampoint{\pgfpoint{2.9cm}{-1.6cm}}
  \pgfplotstreampoint{\pgfpoint{4cm}{-1.2cm}}
  \pgfplotstreampoint{\pgfpoint{5cm}{-1.1cm}}
  \pgfplotstreamend 
  \pgfusepath{stroke}
       
  \pgfsetlinewidth{1pt}
  \color{red}
  \pgfsetplottension{.75}
  \pgfplothandlercurveto
  \pgfplotstreamstart
  \pgfplotstreampoint{\pgfpoint{-5cm}{0cm}}  
  \pgfplotstreampoint{\pgfpoint{-4cm}{-.7cm}}   
  \pgfplotstreampoint{\pgfpoint{-3cm}{-.5cm}}
  \pgfplotstreampoint{\pgfpoint{-2cm}{-.6cm}}
  \pgfplotstreampoint{\pgfpoint{-1cm}{-.55cm}}
  \pgfplotstreampoint{\pgfpoint{0cm}{-1cm}}
  \pgfplotstreampoint{\pgfpoint{1cm}{-1.9cm}}
  \pgfplotstreampoint{\pgfpoint{2cm}{-1.2cm}}
  \pgfplotstreampoint{\pgfpoint{2.4cm}{-1cm}}

  \pgfplotstreamend 
  \pgfusepath{stroke}     

  \node[vertex, red] at (2.4, -1.05){};
  \draw [-,red] (2.4, -1.05) to (a){};
  \draw [red] (a) to  (5.1, 1.22){};
  
  \end{tikzpicture}
\caption{The quasi-geodesic $\alpha'$ (in red) is in the class $\bfc$ which is 
contained in $\calU_{\kappa}(\bfb, \rr)$. Therefore, $\alpha|_\rr= \alpha'|_\rr$ is near $b \in \bfb$.}
\end{figure}
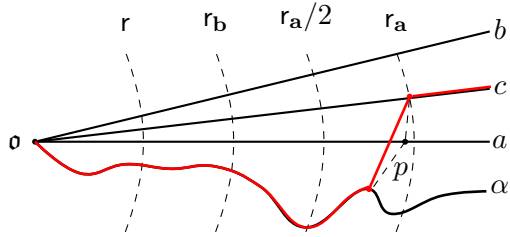

Since $\bfc \in \calU_{\kappa}(\bfb, \rr_\bfb)$, we have
\[
\alpha|_{\rr_\bfb}=\alpha'|_{\rr_\bfb} \subset \calN_\kappa\big(b, \mm_b(9\qq, \sQ)\big).
\]
But $\alpha|_{\rr_\bfb}$ is really a $(\qq, \sQ)$--quasi-geodesic. Hence, 
letting $\nn(\qq, \sQ) = \mm_b(9\qq, \sQ)$, \thmref{Thm:Strong} implies that 
\[
\alpha|_\rr\subset \calN_\kappa\big(b, \mm_b(\qq, \sQ)\big).
\]
But this holds for every such $\alpha$, thus $\bfa \in \calU_{\kappa}(\bfb, \rr)$. 
This finishes the proof. 
\end{proof}

\begin{remark} \label{Rem:r} 
Let $\phi \from \pka X \times \RR \to \RR$ be a map so that $\rr_\bfb = \phi(\bfb, \rr)$
as above. We can define a similar map for $\rr_\bfa$. Note that, in either part 
of \propref{Prop:Normal}, the radius $\rr_\bfa$ does not really depend on $\bfb$ or $\rr$. 
It depends on $\bfa$, $\rr_\bfb$ and the maximum value of $\qq$ and $\sQ$ so that 
$\mm_b(\qq, \sQ)$ is small compared to $\rr$. But such an upper-bound always exists,
for example, $\qq, \sQ \leq \mm_b(\qq, \sQ) \leq \rr \leq \rr_\bfb$. 
Hence, there are maps $\psi_1, \psi_2 \from \pka X \times \RR \to \RR$
where $\rr_\bfa = \psi_1(\bfa, \rr_\bfb)$ in the first part of \propref{Prop:Normal}
and $\rr_\bfa = \psi_2(\bfa, \rr_\bfb)$ in the second part. These maps make the
dependence of constants more clear and we will refer to these map in the proof of 
\thmref{Thm:Metrizable}. Using this notation, \propref{Prop:Normal}
can be written as
\begin{equation} \label{Eq:psi_1}
\bfa \in \calU_{\kappa}\big(\bfb, \phi(\bfb, \rr)\big) 
\qquad\Longrightarrow\qquad
\calU_{\kappa}\Big(\bfa, \psi_1\big(\bfa, \phi(\bfb, \rr) \big) \Big) 
   \subset \calU_\kappa(\bfb, \rr).
\end{equation} 
and 
\begin{equation} \label{Eq:psi_2}
\calU_{\kappa}\Big(\bfa, \psi_2\big(\bfa, \phi(\bfb, \rr)   \big) \Big) 
   \cap \calU_{\kappa}\big(\bfb , \phi(\bfb, \rr)\big) \not= \emptyset. 
\qquad\Longrightarrow\qquad
\bfa  \in \calU_{\kappa}(\bfb, \rr).
\end{equation} 
\end{remark}

\subsection*{A fundamental system of neighborhoods}
We will show that the sets $\calU_{\kappa}(\bfb, \rr)$ form a fundamental system of
neighborhoods for $\partial_\kappa X$ that can be used to define a topology on 
$\pka X$. For $\bfb \in \partial_\kappa X$, define 
\[
\calB(\bfb) = \Big\{ \calV \subset \partial_\kappa X \ST 
\calU(\bfb, \rr) \subset \calV \quad \text{for some $\rr>0$}  \Big\}. 
\]
We would like to equip $\pka X$ with a topology where $\calB(\bfb)$ is the set of 
neighborhoods of $\bfb$. Recall that $\calV$ is a neighborhood of $\bfb$ if it contains
an open set that includes $\bfb$. We need to check that $\calB(\bfb)$ has 
certain properties. 

\begin{lemma}
For every $\bfb \in \partial_s X$, the set $\calB(\bfb)$ satisfies the following properties:
\begin{enumerate}[(i)]
\item Every subset of $\pka X$ which contains a set belonging to $\calB(\bfb)$ 
itself belongs to $\calB(\bfb)$.
\item Every finite intersection of sets of $\calB(\bfb)$ belongs to $\calB(\bfb)$.
\item The element $\bfb$ is in every set of $\calB(\bfb)$.
\item If $\calV \in \calB(\bfb)$ then there is $\calW \in \calB(\bfb)$ such that, 
for every $\bfa \in \calW$, we have $\calV \in \calB(\bfa)$.
\end{enumerate}
\end{lemma}

\begin{proof}
Property (i) is immediate from the definition of $\calB(\bfb)$. To see (ii), consider
sets $\calV_1, \dots, \calV_k \in \calB(\bfb)$ and let $\rr_i$ be such that 
$\calU_{\kappa}(\bfb, \rr_i) \subset \calV_i$ and let $\rr = \max \rr_i$. 
Note that $\calU_{\kappa}(\bfb, \rr)  \subset \calU_{\kappa}(\bfb, \rr_i)$ by definition. Therefore, 
\[
\calU_{\kappa}(\bfb, \rr) \subset \bigcap_i \calV_i
\]
and hence the intersection is in $\calB(\bfb)$. Property (iii) holds since, 
by \corref{Cor:m_b}, every $(\qq, \sQ)$--quasi-geodesic $\beta \in \bfb$ lies inside 
$\calN_\kappa \big(b, \mm_b(\qq, \sQ)\big)$ and hence $\bfb \in \calU_{\kappa}(\bfb, \rr)$ for every $\rr$. 
Property (iv) follows from the first part of \propref{Prop:Normal}. 
\end{proof}

These properties for $\calB(\bfb)$ are characteristic of the set of neighborhoods
of $\bfb$. That is, 

\begin{proposition}[\cite{Bourbaki} Proposition 2]
If to each elements $\bfb \in \pka X$ there corresponds a set $\calB(\bfb)$ of
subsets of $\pka X$ such that properties (i) to (iv) above are satisfied, then there is a
unique topological structure on $\pka X$ such that for each $\bfb \in \pka X$, 
$\calB(\bfb)$ is the set of neighborhoods of $\bfb$ in this topology. 
\end{proposition}

We now equip $\pka X$ with this topological structure. Then a set 
$\calW \subset \pka X$ is open if for every $\bfb \in \calW$ there is 
$\rr>0$ such that $\calU_{\kappa}(\bfb, \rr) \subset \calW$. We refer to this topology as the 
\emph{visual topology on quasi-geodesics} and from now on we consider
 $\partial_\kappa X$ to be a topological space. 

\subsection*{Properties of the topology}
In this section, we establish some topological properties of $\pka X$. 
We will show that $\pka X$ is metrizable and, for $\kappa' \prec \kappa$, we 
show that the inclusion $\partial_{\kappa'} X \subset \pka X$ is a topological embedding. 

We make use the following criterion for a topological space to be metrizable. 

\begin{theorem}[Theorem 3, \cite{Frink}] \label{Thm:Metrizable-condition} 
Assume, for every point $\bfb$ of a topological space, there exists a monotonic 
decreasing sequence $\calU_1(\bfb), \calU_2(\bfb), \cdots, \calU_i(\bfb), \cdots$
of neighborhoods whose intersection is $\bfb$ and such that the following holds: 
For every point $\bfb$ of the neighborhood space and every integer $i$, there exists 
an integer $j=j(\bfb,i) >i$ such that if $\bfa$ is any point for which $\calU_j(\bfa)$ and 
$\calU_j(\bfb)$ have a point in common then $\calU_j(\bfa) \subset \calU_i(\bfb)$. 
Then the space is homeomorphic to a metric space.
\end{theorem}

We check this condition for $\pka X$. 

\begin{theorem}\label{Thm:Metrizable}
The space $\pka X$ is metrizable. 
\end{theorem}

\begin{proof}
Recall the maps $\phi, \psi_1, \psi_2 \from \pka X \times \RR \to \RR$ from 
\remref{Rem:r}. For $i \in \NN$ and $\bfa \in \pka X$, define
\[
\calU_i(\bfa) = \calU_{\kappa}(\bfa, \rr_i(\bfa)),
\qquad\text{where}\qquad
\rr_i(\bfa)=\max\big(i, \psi_1(\bfa, i), \psi_2(\bfa, i)\big).
\]
Also, given $\bfb$ and $i$, we define 
\[
j= j(\bfb, i) = \Big\lceil \phi\big( \bfb, \phi(\bfb, \rr_i(\bfb))\big)\Big\rceil. 
\]

Assume $\calU_j(\bfa)$ and $\calU_j(\bfb)$ have a point in common, that is, 
\[
\calU_{\kappa}\big(\bfa, \rr_j(\bfa)\big) \cap \calU_{\kappa}\big( \bfb, \rr_j(\bfb)\big) \not = \emptyset. 
\]
Since, 
\[
\rr_j(\bfa) \geq \psi_2(\bfa, j) \geq \psi_2\Big(\bfa, \phi\big(\bfb, \phi(\bfb, \rr_i(\bfb))\big) \Big)
\qquad\text{and}\qquad
\rr_j(\bfb) \geq j \geq \phi(\bfb, \phi(\bfb, \rr_i(\bfb))) 
\]
\eqnref{Eq:psi_2} implies
\[
\bfa \in \calU_{\kappa}\Big(\bfb, \phi\big(\bfb,\rr_i(\bfb)\big)\Big). 
\]
Now, \eqnref{Eq:psi_1} implies 
\[
\calU_{\kappa} \Big(\bfa, \psi_1\big(\bfa, \phi(\bfb, \rr_i(\bfb)) \big) \Big) 
\subset \calU_{\kappa}\big(\bfb, \rr_i(\bfb)\big). 
\]
But
\[
\rr_j(\bfa) \geq  \psi_1(\bfa, \phi(\bfb, \phi(\bfb, \rr_i(\bfb))) ) 
\geq \psi_1(\bfa, \phi(\bfb, \rr_i(\bfb)) ). 
\]
Therefore, 
\[
\calU_{\kappa} \big(\bfa, \rr_j(\bfa)\big) \subset \calU_{\kappa}\big(\bfb, \rr_i(\bfb)\big). 
\]
Which is to say $\calU_j(\bfa) \subset \calU_i(\bfb)$. 
The theorem follows from \thmref{Thm:Metrizable-condition}. 
\end{proof}

Lastly, we prove that different boundaries associated with different sublinear functions are nested.

\begin{proposition}\label{subspacetopology}
Let $\kappa, \kappa'$ be sublinear functions such that, for some $\sM>0$, 
\begin{equation}\label{nested}
\kappa'(t)  \leq \sM \cdot \kappa(t), \qquad \forall t>0. 
\end{equation}
Then, $\partial_{\kappa'} X \subseteq \partial_{\kappa} X$ as a subspace with 
the subspace topology.
\end{proposition}

\begin{proof}
It is immediate from the definition that $\partial_{\kappa'} X$ is a subset of 
$\partial_{\kappa} X$. First we have to show that the intersection of an open set in 
$\partial_{\kappa}X$ with $\partial_{\kappa'} X$ is open in $\partial_{\kappa'} X$. 

Let $\calV$ be an open set in $\pka X$ and consider 
$\bfb \in \calV \cap \partial_{\kappa'} X$. Let $\mm_b$ be the $\kappa$--Morse gauge 
for $b$ and let $\mm_b'$ be the $\kappa'$--Morse gauge for $b$. 
Let radius $\rr>0$ be such that  $\calU_{\kappa}(\bfb, \rr) \subset \calV$. 
We need to find radius $\sR$ so that 
$\calU_{\kappa'}(\bb, \sR) \subset \calU_{\kappa}(\bfb, \rr)$. 
For any $\qq, \sQ$, where $\mm_b(\qq, \sQ)$ is small compared to $\rr$, 
there is $\sR=\sR(b, \rr, \mm'_b, \kappa'(\qq, \sQ))$ as in \thmref{Thm:Strong}. 
We denote the maximum such radius again with $\sR$. 

Let $\bfa \in \calU_{\kappa'}(\bb, \sR)$ and let $\alpha \in \bfa$ be a 
$(\qq, \sQ)$--quasi-geodesic such that $\mm_b(\qq, \sQ)$ is small compared to $\rr$. 
Taking $\sR$ even larger if needed, we can assume that $\mm_b'(\qq, \sQ)$ is small 
compare to $\sR$. Then, $\bfa \in \calU_{\kappa'}(\bb, \sR)$ implies that
\[
d_X(\alpha_\sR, b) \leq \mm_b'(\qq, \sQ) \cdot \kappa'(\sR). 
\]
By \thmref{Thm:Strong},
\[
\alpha|_{\rr} \subset \calN_\kappa\big(b, \mm_b(\qq, \sQ)\big). 
\]
Since this holds for every such $\alpha \in \bfa$, we have 
$\bfa \in \calU_\kappa(\bfb, \rr)$. Therefore, 
\[
\calU_{\kappa'}(\bfb, \sR) \subset \calU_\kappa(\bfb, \rr) \subset \calV.
\] 
That is, every such point $\bfb$ is in the interior of $\calV \cap \partial_{\kappa'} X$
and $\calV \cap \partial_{\kappa'} X$ is open in $\partial_{\kappa'} X$.

Next we show that every open set in $\partial_{\kappa'} X$ is the intersection of an open set of 
$\partial_{\kappa} X$ with $\partial_{\kappa'} X$. It suffices to show that given an open set 
$\calV' \subset \partial_{\kappa'} X$, and a point $\bfc \in \calV'$, there exists a neighbourhood 
$\calU_{\kappa}(\bfc, \rr)$ such that $\calU_{\kappa}(\bfc, \rr) \cap \partial_{\kappa'} X \subset \calV'$. 
By definition of the topology there exists an open set $\calU_{\kappa'}(\bfc, \rr_{c}) $ such that 
\[ \bfc \in \calU_{\kappa'}(\bfc, \rr_c) \subset \calV'.\]

 Now by Theorem~\ref{Thm:Strong}, there exists $\sR_c$ such that for any $(\qq, \sQ)$--quasi-geodesic 
 $\eta$ where $(\qq, \sQ)$ is small compared to $\rr_c$,
\[
d_X\big(\eta(t_\sR), c\big) \leq  \kappa(\sR_c)
\quad\Longrightarrow\quad
\eta[0, t_{\rr_{c}}] \subset \calN_{\kappa'}\big(c, \mm_c(\qq, \sQ)\big). 
\]
That is to say, 
\[
\calU_{\kappa}(c, \sR) \cap \partial_{\kappa'} X\subset \calU_{\kappa'}(c, \rr_{c})
\]
Let $\calW_c$ be the interior of $\calU_{\kappa}(c, \sR)$ which is an open set in $\partial_\kappa X$
and still contains $\bfc$. We have
\[
\calW_c \cap \partial_{\kappa'} X\subset \calU_{\kappa'}(c, \rr_{c}),
\]
and therefore, 
\[
 \calV' = \bigcup_{ \bfc \in \calV'} (\calW_c \cap \partial_{\kappa'} X )
  =\left(\bigcup_{ \bfc \in \calV'} \calW_c \right) \cap \partial_{\kappa'} X.
\]
This finishes the proof. 
\end{proof}

\section{Boundary of a \CAT group}
Let $G$ be a finitely generated group that acts geometrically on $X$, that is, properly 
discontinuously, co-compactly and by isometries. Let $\go$ denote the base-point of $X$.
Equip $G$ with the word length associated to some generating set. Also, 
given an element $g \in G$, denote the image of $\go$ under the action of 
$g$ by $g \go$. Then the map 
\[
\Psi \from G \to X, \qquad \Psi(g) = g \go 
\]
defines a quasi-isometry between $G$ and $X$ which means there is an association 
between quasi-geodesics in $G$ and in $X$. Hence, we can define 
$\pka G$ to be $\pka X$. Namely, consider a path $P=\{ g_i\}_{i=0}^\infty$ 
in $G$ such that $g_0 = id$ and $g_i$ and $g_{i+1}$ differ by a generator. 
Define $\beta_P$ to be the ray in $X$ that is a concatenation of 
geodesic segments $[g_i, g_{i+1}]$. If $\beta_P$ is a $\kappa$-Morse 
quasi-geodesic in $X$, then we say $g_i \to [\beta_P]$. In other words, 
$\pka G$ is the set of $\kappa$--equivalence classes of quasi-geodesic rays in $G$ 
so that the associated quasi-geodesic in $X$ is $\kappa$-Morse. 

However, $G$ may act geometrically on different \CAT spaces. To show 
$\pka G$ is well defined, we need to show different such spaces give 
the same boundary for $G$. We show, more generally, that $\pka X$ is invariant 
under quasi-isometry. 

\begin{theorem}\label{invarianttopology}
Consider proper \CAT metric spaces $X$ and $Y$ and let $\Phi \from X \to Y$ be a 
$(\kk, \sK)$--quasi-isometry. Then $\Phi$ induces a homeomorphism 
$\Phi^\star \from \partial_\kappa X \to \partial_\kappa Y$ 
for every sublinear function $\kappa$ where, for $\bfb \in \pka X$ and $\beta \in \bfb$, 
\[
\Phi^\star(\bfb) = [\Phi \circ \beta]. 
\]
\end{theorem}

\begin{proof}
For a quasi-geodesic ray $\zeta \from [0, \infty) \to X$ in $X$ let 
$\Phi \zeta$ be a quasi-geodesic ray in $Y$ constructed from the composition of 
$\zeta$ and $\Phi$ as in \defref{Def:Quadi-Geodesic}. It is immediate
from the definition that two quasi-geodesics $\zeta$ and $\xi$ in $X$ $\kappa$--fellow 
travel each other if and only if $\Phi \zeta$ and $\Phi \xi$ $\kappa$--fellow travel 
each other in $Y$. Also (again immediate from the definition) the property of being 
$\kappa$--Morse is preserved under a quasi-isometry. Hence, 
$[\zeta] \in  \partial_\kappa X$ if and only if $[\Phi \zeta] \in \partial_\kappa Y$.  
Therefore, $\Phi^\star$ defined as above gives a bijection between $\partial_\kappa X$ 
and $\partial_\kappa Y$. We need to show that $(\Phi^\star)^{-1}$ is continuous. Then, the 
same argument applied in the other direction will show that $\Phi^\star$ is also continuous 
which means $\Phi^\star$ is a homeomorphism. 

Let $\calV$ be an open set in $\pka X$, $\bfb_X \in \calV$ and 
$\calU_{\kappa}(\bfb_X, \rr)$ be a neighborhood $\bfb$ that is contained in $\calV$. 
Let $\bfb_Y = \Phi^\star(\bfb_X)$. We need to show that there is a constant $\rr'$ 
such that, for every point $\bfa_Y \in \calU_\kappa(\bfb_Y, \rr')$, we have 
\[
\bfa_X = (\Phi^\star)^{-1}(\bfa_Y) \in \calU_\kappa(\bfb_X, \rr).
\]
Let $\qq'$ and $\sQ'$ be constants (depending on $\qq, \sQ, \kk$ and $\sK$)
such that if $\zeta$ is a $(\qq, \sQ)$--quasi-geodesic where
$\mm_{b_X}(\qq, \sQ)$ is small compared to $\rr$ then $\Phi \zeta$ is a 
$(\qq', \sQ')$--quasi-geodesic. Let $b_X$ be the unique geodesic ray in $\bfb_X$, 
let $b_Y$ be the unique geodesic ray in $\bfb_Y$ and let $\mm_{b_X}$ and 
$\mm_{b_Y}$ be their Morse gauges respectively. 
By \corref{Cor:m_b}, there is
a constant $\nn_1$ depending on $\kk$, $\sK$ and $\mm_{b_Y}$ such that 
\[
\Phi b_X \subset \calN_\kappa(b_Y, \nn_1). 
\]
For 
\[
\nn = \kk \big( \mm_{b_Y}(\qq', \sQ') + \nn_1\big) (\kk + \sK) + \sK
\]
let $\sR = \sR(b_X, \rr, \nn, \kappa)$ as in \thmref{Thm:Strong} and 
choose $\rr'$ such that $\rr' \geq \kk \, \sR + \sK$ and $\mm_{b_Y}(\qq', \sQ')$ is small 
compare to $\rr'$.

Let $\alpha \in \bfa_X$ be a $(\qq, \sQ)$--quasi-geodesic where 
$\mm_{b_X}(\qq, \sQ)$ is small compared to $\rr$ such that 
$\Phi \alpha \in \calU_{\kappa}(\bfb_Y, \rr')$. By our choice of $\rr'$, 
$\mm_{b_Y}(\qq', \sQ')$ is small compared to $\rr'$. Hence, 
\[
\Phi \alpha|_{\rr'} \subset \calN_\kappa\big(b_Y, \mm_{b_Y}(\qq', \sQ')\big)
\] 
Pick $x \in \alpha_X|_{\sR}$. Then $\Phi x \in \Phi \alpha|_{\rr'}$ and we have 
\begin{align*} 
d_X(x, b_X) &\leq \kk (d_Y(\Phi(x), \Phi b_X) + \sK\\
& \leq \kk \Big(d_Y(\Phi(x), b_Y) + \nn_1 \cdot \kappa(\Phi x) \Big) + \sK\\
& \leq \kk \big( \mm_{b_Y}(\qq', \sQ') + \nn_1\big) \cdot \kappa(\Phi x) + \sK
\end{align*} 
This and 
\[
\kappa(\Phi x) \leq \kk \kappa(x) + \sK \leq (\kk + \sK) \kappa(x)
\]
imply that 
\[
\alpha|_{\sR} \subset \calN_\kappa(b_X, \nn). 
\]
Now, \thmref{Thm:Strong} implies that 
\[
\alpha|_\rr \subset \calN_\kappa(b_X, \mm_{b_X}). 
\]
Therefore, $\bfa_X \in \calU_{\kappa}(\bfb_X, \rr)$ and
\[
(\Phi^\star)^{-1} \calU_{\kappa}(\bfb_Y, \rr') \subset \calU_{\kappa}(\bfb_X, \rr). 
\]
But $ \calU_{\kappa}(\bfb_Y, \rr') $ contains an open neighborhood of $\bfb_Y$, therefore, 
$\bfb_Y$ is in the interior of $\Phi \calV$. This finishes the proof. 
\end{proof}

\section{Examples}
\subsection*{A tree of flats} \label{Sec:example}
In this section we examine $\kappa$--boundaries of a few simple examples to illustrate 
several typical properties of $\kappa$--boundaries of \CAT spaces. 
Consider the right-angled Artin group
\[
A = \ZZ^{2} \ast \ZZ = \Big\langle g_{1} ,g_{2} ,g_{3} \ST [g_{1}, g_{2}] \Big\rangle.
\]
Let $X_A$ be the universal cover of the Salvetti complex of $\ZZ^{2} \ast \ZZ$, or simply the 
\emph{universal Salvetti complex}, as in Definition~\ref{salvetti}. 
We observe that $X_A$ is a tree of flats. The flats are associated to orbits of
conjugate copies of the subgroup 
\[\big \langle g_{1}, g_{2} \st  [g_{1}, g_{2}] \big\rangle \simeq \ZZ^2. \] 
The oriented edges that are outside of these flats are labelled $g_{3}$. 

We equip $X_A$ with a metric so that each flat is isometric to the Euclidean
plane $\EE^2$, the axes of $g_{1}$ and $g_{2}$ intersect 
at a 90-degree angle and edges labeled $g_3$ are attached at the lattice points. 
The space is simply connected and the metric on $X_A$ is \CAT. 
The closest-point projection of any flat to any other flat is a single point. 
Also, since flats are convex subspaces, given a geodesic ray in $X_A$, 
there is a well-defined \emph{itinerary of flats} that the geodesic passes through. 
Choose a base-point $\go$ where an edge labeled $g_3$ is attached to
a flat and let $Y_0$ be the flat that contains $\go$. 
As before, we always assume a geodesic ray starts at $\go$. 
 
We give a characterization of the $\kappa$-contracting rays in $X_A$. 
First we need the following lemma:

\begin{lemma}\label{uniqueflat}
Let $b$ be a geodesic ray in $X_A$. Given any ball  $B$ disjoint from $b$. 
The projection $\pi_b(B)$ of $B$ to $b$ lies inside a unique flat.
\end{lemma}

\begin{proof}
Assume for contradiction that $\pi_b(B)$ contains a point $w$ in the interior of an 
edge $e = (v_{1}, v_{2})$ labelled by $g_{3}$. Then $b$ traverses $e$.
The point $w$ is a cut-point of $X_A$. Let $v_1$ be the vertex of the edge
$e$ that is in the same component as the center of the ball $B$.  
Since $\pi_b(B)$ contains a point $w$, there exists a point $x \in B$ where $w= x_b$. 
However, we have 
\begin{equation}\label{noedge}
d(x, w)  > d(x, v_{1}) \geq d(x, b).
\end{equation}
This contradicts the assumption that $w$ is a nearest-point projection from $x$ to $b$. 
This holds for every edge labeled $g_3$. Hence $\pi_b(B)$ in contained a single flat.
\end{proof}

\begin{lemma}\label{characterization}
A unit speed geodesic ray $b$ in $X_A$ is $\kappa$--contracting, if and only if, 
there exists a constant $\cc$ such that if $b[t_1, t_2]$ is contained in a flat, then
\[ |t_{1}- t_{2}| \leq \cc \cdot \kappa(t_1). \] 
\end{lemma}

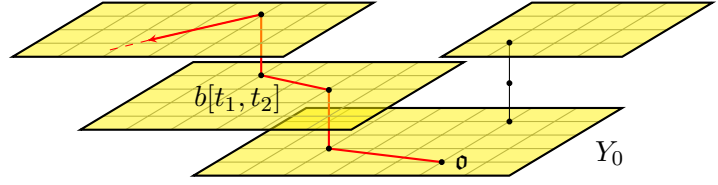
\begin{figure}[H]
\begin{tikzpicture}[scale=0.3]
\tikzstyle{vertex} =[circle,draw,fill=black,thick, inner sep=0pt,minimum size=.5 mm]
 
[thick, 
    scale=1,
    vertex/.style={circle,draw,fill=black,thick,
                   inner sep=0pt,minimum size= .5 mm},
                  
      trans/.style={thick,->, shorten >=6pt,shorten <=6pt,>=stealth},
   ]

\begin{scope}
        \myGlobalTransformation{-57}{0};
        \draw [black!50, step=2cm] grid (14, 10);
         \draw [thick, draw=black, fill=yellow, fill opacity=0.5]
       (0,0) -- (0,10) -- (14,10) --(14,0)--cycle;
\end{scope}

\draw[thick, red] (9, 0.6)--(4, 1.2)--(4, 3.8){};

\node [vertex] at (9, 0.6)[label=right:$\go$] {}; 
\node  at (18, 1)[label=left:$Y_{0}$] {}; 

\draw (12, 2.4) to (12, 5.9){};

\begin{scope}
        \myGlobalTransformation{256}{150};
        \draw [black!50, step=2cm] grid (8, 8);
        \draw [thick, draw=black, fill=yellow, fill opacity=0.5]
       (0,0) -- (0,8) -- (8,8) --(8,0)--cycle;   
\end{scope}

\node [vertex] at (12, 2.4){};
\node [vertex] at (12, 4.1){};
\node [vertex] at (12, 5.9){};


\begin{scope}
        \myGlobalTransformation{-171}{75};
        \draw [black!50, step=2cm] (0, -2) grid (12, 8);
        \draw [thick, draw=black, fill=yellow, fill opacity=0.5]
       (0,-2) -- (0,8) -- (12,8) --(12,-2)--cycle;
\end{scope}

\draw[thick, red] (4, 3.8)--(1, 4.45)--(1, 7.15){};

\node at (2.8, 3.4) [label=left:${b[ t_1, t_2] }$] {}; 

\node [vertex] at (4, 1.2){};
\node [vertex] at (4, 3.8){};


\begin{scope}
        \myGlobalTransformation{-285}{150};
        \draw [black!50, step=2cm] grid (12, 8);
        \draw [thick, draw=black, fill=yellow, fill opacity=0.5]
       (0,0) -- (0,8) -- (12,8) --(12,0)--cycle;
\end{scope}

\draw[thick, red] (1, 7.15)--(-4, 6){};

\draw[->, red] (1, 7.15)--(-4, 6){};
\draw[dashed, red] (-4, 6)--(-6, 5.5) {};

\node [vertex] at (1, 4.45){};
\node [vertex] at (1, 7.15){};

\end{tikzpicture}
\end{figure}

\begin{proof}
First we consider the ``if'' direction. Let $\{Y_i \}$ be 
the sequence of flats visited by $b$. By Lemma~\ref{uniqueflat}, if a ball $B$ is disjoint 
from $b$ then $\pi_b(B)$ is contained in some $Y_i$. Let $[t_1, t_2]$ be the interval 
of time where the image of $b$ is in $Y_i$. 

Let $x$ be the center of the ball $B$ and $y$ be any other point in $B$. 
By Lemma~\ref{uniqueflat}, $x_{b} = b(t)$ for $t \in [t_1, t_2]$. Therefore, 
 \[
\Norm x \geq \Norm{x_b} \geq t_{1}.\]
Thus we have 
\[
d_{X_A}(x_b, y_b) \leq |t_{2} - t_{1}| \leq \cc \cdot \kappa(t_{1}),
\]
which means $b$ is $\kappa$--contracting and we can set $\cc_b=\cc$. 
 
For the ``only if'' direction, assume $b$ is $\kappa$-contracting with $\cc_b$ as a contracting
constant. For an interval $b[t_1, t_2]$ that stays in a flat, consider the ball $B$ whose 
center $x$ is at a distance $|t_{2} - t_{1}|$ from the point $b(t_1)$ in a perpendicular direction
from the segment $b[t_1, t_2]$ and with radius $(t_2 -t_1)$. Then 
$\pi_b(B)=b[t_1, t_2]$. The definition of $\kappa$-contracting geodesic ray dictates that 
 \[ |t_2 - t_1| \leq \cc_{b} \cdot \kappa(x) \leq \cc_{b} \cdot \kappa(t_{1}+(t_2-t_{1})) = \cc_{b} \cdot \kappa(t_2) . \]
 By Lemma~\ref{Lem:sublinear-estimate}, 
 \[
 \kappa(t_2) \leq \cc' \cdot \kappa(t_1),
 \]
for some $\cc'$ depending on $\kappa$ and $\cc_b$. 
Thus we have 
\[ |t_2 - t_1| \leq \cc_{b} \cdot \kappa(t_2) \leq \cc_{b} \cc' \cdot \kappa(t_1).\]
 \end{proof}

\begin{proposition}\label{Prop:strict}
If $\kappa(t), \kappa'(t)$ are two sublinear functions such that,
\[
\lim_{t \to \infty} \frac{\kappa'(t)}{\kappa(t)} = 0
\]
Then $\partial_{\kappa'} X_A \subsetneq \partial_{\kappa} X_A$, that is to say, 
$ \partial_{\kappa} X_A$ strictly contains $\partial_{\kappa'} X_A$.
\end{proposition}
\begin{proof}
The fact that $\partial_{\kappa'} X \subseteq \partial_{\kappa} X$ follows 
Proposition~\ref{subspacetopology}. We give a specific construction of a geodesic ray 
$b$ that is in $\pka X$ but not $\partial_{\kappa'} X$. The ray $b$ is the concatenation
of vertical segments $v_i$ consisting of edges labeled $g_3$ and horizontal segments 
$h_i$ that are contained in a single flat. 

Let $i_0$ be an integer so that $2^{i_0} \geq \kappa(2^{i_0})$ and let $b_{i_0}$ 
be a vertical segment of length $2^{i_0}$. For $i>i_0$, assume a segment $b_{i-1}$
is given. Continue $b_{i-1}$ along a horizontal segment $h_i$ of length 
$\lfloor \kappa(2^{i}) \rfloor$, then along a vertical segment $v_i$ of length 
$\lceil 2^{i} -\kappa(2^{i})\rceil$ and denote the resulting segment by $b_i$. 
We see inductively that $|b_i|= 2^{i-1}$ because, 
\[
|b_i| = |b_{i-1}| + |h_i|+ |v_i|  = 
2^i + \lfloor \kappa(2^{i}) \rfloor + 
\lceil 2^{n} -\kappa(2^{i})\rceil = 2^{i+1}. 
\]
Also, 
\[
\lfloor \kappa(2^{i}) \rfloor = |h_i| \leq \kappa(|b_{i-1}|). 
\]
That is, if we let the ray $b$ be the union of the segments $b_i$, then $b$ 
satisfies \lemref{characterization} for the sublinear function $\kappa$ but not for
$\kappa'$. Hence, $[b] \in (\pka X - \partial_{\kappa'} X)$. 
\end{proof}

As an easy consequence, we have

\begin{corollary}
There exists two \CAT spaces that are not distinguishable by their Charney-Sultan 
contracting boundaries \cite{contracting} but are distinguishable by their sublinear Morse boundaries.
\end{corollary}

\begin{proof}
We can adjust the metric on $X_A$ by changing the lengths of the edges.
We say a flat $Y$ is at \emph{height $n$}, if the geodesic segment connecting $\go$
to any point in $Y$ traverses through $n$ edges labeled by $g_{3}$ (in either direction). 
Let $X_{\sqrt \param}$ be the space obtained from $X_A$ where the side lengths of unit 
squares in flats at height $n$ is scaled to $\sqrt{n}$. Similarly, let $X_{\log}$ be 
the space obtained from $X_A$ the side lengths of unit squares in flats at height $n$ 
is scaled to $\log(n)$. 

Since $\sqrt{n}$ grows faster than $\log n$, the $\log$--boundary of $X_{\sqrt{\param}}$ is 
a set that contains geodesic rays that eventually cannot travel even one edge in any flat. That is to say, after a finite time, this geodesic ray will travel along the $g_{3}$ direction only. The number of such geodesic rays is countable. 

On the other hand, we see from Lemma~\ref{characterization} that the $\log$--boundary 
of $X_{\log}$ consists of geodesic rays whose projections to any flat are bounded by 
$\log$ of the time they enter the flat. A geodesic in this boundary therefore can travel in 
infinitely many flats. Therefore, the $\log$--boundary of $X_{\log}$ is an uncountable set.

Meanwhile, the Morse boundaries of both $X_{\sqrt{\param}}$ and  $X_{\log}$ consist of 
geodesic rays that  eventually travel along the $g_{3}$ direction only. By the previous 
argument the Morse boundary of $X_{\sqrt{n}}$ and the Morse boundary of $X_{\log}$ 
are homeomorphic via the equivariant map of the group.
\end{proof}

\subsection*{Random Walks}

As mentioned in the introduction, one motivation for constructing 
the $\kappa$--Morse boundary is to study random walks on a group. (For details of 
construction of random walks on groups, see the Appendix.) In the setting
of the group $A$ acting on $X_A$, Theorem~\ref{T:log-exc} tells us that, for
almost every sample path in $X_A$, the maximum amount of time spent on a given 
flat after $n$ steps is bounded by $\cc \cdot \log n$. 

Furthermore, since $X_A$ is \CAT, by \cite{Kar-Mar}, almost every sample path $w=\{ w_n\}$ tracks 
a geodesic ray in $X_A$ which we denote $b_w$. That is, there is $\uu>0$ such that 
the distance between $w_n(\go)$ and $b(\uu \cdot n)$ grows sublinearly with $n$. 
Therefore, for every flat $Y$, the projection of $w_n(\go)$ to $Y$ is eventually the same 
as the projection of $b(\uu \cdot n)$ to $Y$, which is the point in $Y$ where $b_w$ exits 
$Y$. In fact, if $n$ is larger than a fixed multiple of $d_{X_A}(\go, Y)$, then $w_n(\go)$ is 
closer to $b_w$ than to $Y$, and hence the path connecting $w_n(\go)$ to $b_w$ is 
disjoint from $Y$ and projects to a point in $Y$. 

By the above theorem, $d_{Y} \big(\pi_{Y}(1), \pi_{Y}(w_{n})\big)$ grows only 
logarithmically. Hence, the time $b_w$ spends in $Y$ is less than
a multiple of the distance between $Y$ and $\go$. That is, $b_w$ satisfies the 
condition of \lemref{characterization} and hence $[b_w] \in \partial_{\log} X_A$. 

By  \cite{hitting}, the visual boundary of the Salvetti complex of a right-angled Artin group together with the hitting measure constitutes a 
metric model for the Poisson boundaries of the group. Since 
almost every sample path converges to a point in $\partial_{\log} X_A$,
we have:

\begin{corollary}
Let $\mu$ be a symmetric, finitely supported probability measure on $A = \ZZ \ast \ZZ^{2}$. 
Then $\partial_{\log} X_A$ is a metric model for the Poisson boundary 
$(\ZZ^{2} \ast \ZZ , \mu)$. 
\end{corollary}

In the Appendix, this is generalized to the class of all right-angled Artin groups. 

\subsection*{Other topological properties of $\pka X_A$}
We have shown that $\pka X$ is metrizable which implies that it is, Hausdorff, normal 
and paracompact. However, $\pka X$ is often not compact. For the example
given in this section, we have:

\begin{proposition}\label{totallydisconnected}
The topological space $\pka X_A$ is non-compact, totally disconnected and 
with no isolated points.
\end{proposition}

\begin{proof}
Let $e$ be any vertical edge in $X$, i.e. an edge labeled $g_{3}$. Define
$\calW(e) $ to be the set all points $\bfb \in \pka X$ where the geodesic ray $b \in \bfb$
traverses $e$. For any $\bfb \in \calW(e)$ and $\rr$ large enough, 
$\calU_{\kappa}(\bfb, \rr) \subset \calW(e)$. This is because, if $\bfa \in  \calU_{\kappa}(\bfb, \rr)$, then
the geodesic ray $a \in \bfa$ stays in a $\kappa$--neighborhood of $b$ for distance $\rr$,
namely $a|_\rr \subset \calN_\kappa\big(b, \mm_b(1,0)\big)$, and hence has to also traverse $e$. 
Therefore, $\calW(e)$ is an open set in $\pka X$. 

But $\pka X - \calW(e)$ is also open because it can be written as a union of sets of the form 
$\calW(e')$. For any $b' \not = b$ in $\pka X$, let $e$ be an edge traversed by $b$ and 
not by $b'$. Then $b \in \calW(e)$ and $b' \in \pka X - \calW(e)$ which are both open. 
Thus, $\pka X$ is totally disconnected. 

All sets $\calW(e)$ are homeomorphic to each other and contain more than one point in 
$\pka X$. Let $\{e_i\}$ be the set of vertical segments along $b$ and let $b_i \in \calW(e_i)$
be a point not equal to $b$. Since $\cap_i \calW(e_i)= b$, we have $b_i \to b$. 
That is, $\pka X$ has no isolated points. 

To see that $\pka X$ is not compact, consider a sequence of geodesics $\{ b_j \}$ where
each $b_j$ leaves the flat $Y_{0}$ at coordinate $(j, 0)$ and then follows the 
$g_{3}$--direction indefinitely. All geodesic rays $b_j$ are $\kappa$--contracting. 
But, the point-wise limit of this sequence is the geodesic that lies in $Y_{0}$
which is not contracting for any $\kappa$. In fact, $b_j$ has no limit point in $\pka X$ 
because if $b_j \to b$ then infinity many $b_j$ have to be contained $\calW(e)$ for some
$e$ along $b$. But this does not hold for any $e$. Therefore, $\pka X$ is not compact. 
\end{proof}

However, the boundary does not always have to be totally disconnected. 
In \cite{Jasoncounterexample}, Behrstock constructed a family of
right-angled Coxeter groups where the Morse boundary is not totally disconnected. 
And, since the Morse boundary is a topological subspace of $\pka X$
(see \lemref{subspacetopology}), the same holds for $\pka X$.

\appendix

\section{Poisson boundaries of Right-angled Artin groups} 

\begin{center}
\textsc{Yulan Qing}\footnote{Shanghai Center for Mathematical Sciences, Fudan University, Shanghai, China, \url{qingyulan@fudan.edu.cn}.} \textsc{and Giulio Tiozzo}\footnote{Department of Mathematics, University of Toronto, Toronto, ON, \url{tiozzo@math.utoronto.ca}.}
\end{center}
\bigskip 

As an application of sublinearly Morse boundaries, we show that when 
$\kappa = \sqrt{t \log t}$, the $\kappa$--Morse boundary of the universal Salvetti complex 
is a model for the Poisson boundary of a right-angled Artin group. This establishes Theorem 
\ref{T:intro-Poisson} in the introduction.

Let $\Gamma$ be a finite graph, and let $A(\Gamma)$ be the right-angled Artin group associated to $\Gamma$, which is defined by the presentation
\[ \A := \big \langle v \text{ is a vertex in }\Gamma \ | \ [v, w] = 1, (v, w) \text{ is an edge in }\Gamma \big \rangle. \]
That is to say, there is an infinite order generator for each vertex, and a pair of generators 
commute if and only if there is an edge between the two corresponding vertices in $\Gamma$. 

Each right-angled Artin group is associated with a cube complex known as its 
\emph{Salvetti complex}, and its universal cover $X(\Gamma)$ 
is a proper CAT(0) space on which $\A$ acts cocompactly. We call $X(\Gamma)$ the 
\emph{universal Salvetti complex}. The main theorem of this appendix is the following.

\begin{theorem} \label{T:Poisson-raag}
Let $\mu$ be a finitely supported, generating measure on an irreducible right-angled Artin group $A(\Gamma)$. Then the $\sqrt{t \log t}$--Morse boundary of  $X(\Gamma)$  is a 
QI-invariant topological model for the Poisson boundary of $(A(\Gamma), \mu)$.
\end{theorem}

Let us now recall some background material and fundamental definitions.

\subsection*{Random walks and the Poisson boundary}

Let $G$ be a countable group of isometries of a metric space $X$, and let $\mu$ be a probability measure on $G$.
A measure $\mu$ is \emph{generating} if the semigroup generated by the support of $\mu$ equals $G$.
We define the \emph{random walk} associated to $(G, \mu)$ as the stochastic process
$$w_n := g_1 \dots g_n$$
where $(g_n)_{n \geq 1}$ is a sequence of $G$-valued i.i.d. random variables, each with distribution $\mu$. Let us fix a base point $x \in X$. 
The sequence $(w_n x)_{n \geq 1}$ is called a \emph{sample path} for the random walk. 

In most interesting situations, almost every sample path converges to a point in a suitable boundary $\partial X$; in that case, 
we define the \emph{hitting measure} $\nu$ on $\partial X$ as 
$$\nu(A) := \mathbb{P}\left( \lim_{n \to \infty} w_n x \in A \right).$$

A function $f : G \to \mathbb{R}$ is $\mu$-harmonic if it satisfies a discrete version of the mean value property; namely, $f(g) = \sum_{h \in G} \mu(h) f(gh)$ for any $g \in G$. We denote the space of bounded, $\mu$-harmonic functions as $H^\infty(G, \mu)$. 
Now, the \emph{Poisson transform} $\Phi : L^\infty(\partial X, \nu) \to H^\infty(G, \mu)$ is defined as $$\Phi(f)(g) := \int_{\partial X} f(g(x)) \ d\nu(x)$$
and the space $(\partial X, \nu)$ is the \emph{Poisson boundary} if $\Phi$ is an isomorphism. 

That is, the Poisson boundary is the natural space where all bounded harmonic functions can be represented. It is well-defined as a measurable $G$-space. 
For groups acting on \CAT metric spaces, an identification of the Poisson boundary is given as follows.

\begin{theorem}[\cite{Kar-Mar},  \cite{hitting}] \label{T:converge}
Let $G$ be a countable group of isometries of a \CAT proper metric space such that its action has bounded exponential growth, 
and let $\mu$ be a nonelementary measure on $G$ with finite first moment. Then: if the drift is zero, the Poisson boundary of $(G, \mu)$ is trivial; 
if the drift is positive, almost every sample path converges to the visual boundary of $X$, and the visual boundary with the hitting measure is a model for the Poisson boundary of $(G, \mu)$. 
\end{theorem}

For more general measures, convergence to the visual boundary has been recently proven in \cite{FLM}. In this appendix, we prove: 

\begin{theorem} \label{T:full-measure}
Let $G = A(\Gamma)$ be an irreducible right-angled Artin group, let $\mu$ be a finitely supported, generating measure on $G$, and let $\nu$ 
be the hitting measure for the corresponding random walk. Then the $\kappa$--Morse boundary with $\kappa(t) =  \sqrt{t \log t}$ is a $G$-invariant subset of the visual boundary of full $\nu$-measure. 
\end{theorem}

Theorem \ref{T:full-measure} and Theorem \ref{T:converge} immediately imply Theorem \ref{T:Poisson-raag}, which is the same as Theorem \ref{T:intro-Poisson} in the introduction.

\subsection*{Background on cube complexes}

For all basic definitions related to right-angled Artin groups and the associated \CAT cube complex $X(\Gamma)$, we follow  \cite{intro}. 

\begin{definition}\label{salvetti}
Associated to a right-angled Artin group $\A$ is an infinite and locally finite cube complex called the \emph{Salvetti complex}, constructed as follows: associated to each vertex of $\A$ is a simple closed loop of unit length. If two vertices form an edge in $\A$ then attach to the two associated loops a square torus generated by the two loops intersecting at a right angle. More generally, given a complete subgraph on $k$ vertices, consider a unit $k$-torus generated by $k$ loops intersecting at right angles. The \emph{universal Salvetti complex} associated to $\A$, denoted as $X(\Gamma)$, is then the universal cover of this tori-complex. Notice that the $0$ and $1$-skeleta of $X(\Gamma)$ are isomorphic, respectively, to the $0$ and $1$-skeleton of the Cayley graph of $\A$ with this specific presentation.

\end{definition}
The universal Salvetti complex $X(\Gamma)$ is a \emph{\CAT cube complex}  \cite{Hag14}, which we discuss now.  A \emph{cube complex} is a polyhedral complex in which the cells
are Euclidean cubes of side length one. The attaching maps are isometries identifying the
faces of a given cube with cubes of lower dimension and the intersection of two cubes is a
common face of each. 
Cubes of dimension $0$, $1$ and $2$ are also referred to as vertices, edges and squares. 
A cube complex is finite dimensional if there is an upper bound on the dimension of its cubes.
Finally, a \emph{CAT(0) cube complex} is a simply connected cube complex in which the link of each vertex is a flag simplicial complex.

\subsection*{Hyperplanes and contact graph}

In a \CAT cube complex, consider the equivalence relation on the set of mid-cubes generated by the rule that two mid-cubes are related if they share a face. 
Then a \emph{hyperplane} $H$ is defined as the union of the mid-cubes in a single equivalence class. Every hyperplane $H$ is a geodesic subspace of $X(\Gamma)$ which separates $X(\Gamma)$ into two components. 
We shall refer to the each of these two components as a \emph{half-space}, and denote them as $\{ H^{+}, H^{-} \}$. Two hyperplanes provide four possible half-space intersections; the hyperplanes
\emph{intersect} if and only if each of these four half-space intersections is non-empty. In contrast, we say two convex subcomplexes $F_{1}, F_{2}$ are \emph{parallel} (and we denote it as $F_1 \sim F_2$) if, given any other hyperplane $H'$,
\[
F_{1} \cap H' \neq \emptyset \Leftrightarrow F_{2} \cap H' \neq \emptyset.
\]

We say a hyperplane $H$ \emph{separates} two hyperplanes $H_1, H_2$ if, given any pair of points $x \in H_1, y \in H_2$, all geodesics connecting $x$ and $y$ have non-empty intersection with $H$. 
Lastly, we say a (combinatorial) geodesic \emph{crosses} a hyperplane $H$ if there exists two consecutive vertices on the geodesic such that one belongs to $H^{+}$ and the other belongs to $H^{-}$. 

Given a finite graph $\Gamma$, a \emph{join} 
$J \subset \Gamma$ is an induced subgraph whose vertices can be partitioned into two sets $A$, $B$ such that all edges of the form $\{ (a, b) \ : \  a \in A, b \in B\}$ are edges of $J$. Recall a right-angled Artin group  $A(\Gamma)$ is \emph{irreducible} if $\Gamma$ itself is not a join.
Let $\calJ$ denote the set of all maximal joins of $\Gamma$, where maximality is defined by containment. 

\begin{remark}\label{joincontainslink}
By definition, every join between a vertex and its link is contained in a maximal join.
\end{remark}

\begin{definition}
The \emph{contact graph} $\calC(X)$ of a \CAT cube complex $X$ is a graph whose vertex set is the set of hyperplanes of $X$. 
Moreover, two vertices are adjacent if the corresponding hyperplanes $H_{1}, H_{2}$ satisfy one of the following:
\begin{itemize}
\item either $H_{1}$ intersects $H_{2}$ nontrivially; or
\item $H_{1}$ and $H_{2}$ are not separated by a third hyperplane.
\end{itemize}
It is known that the contact graph is always hyperbolic (in fact, a quasi-tree  \cite{Hag14}).
\end{definition}

\subsection*{Gates and projections}

Given a point  $x$ and a convex subset $Z$ of $X$, the nearest-point projection of $x$ to $Z$ exists and is unique by \CAT geometry. We denote it as $x_{Z}$. 

\begin{definition}
If $K \subset X$ is convex, then for all $x \in X^{(0)}$, there exists a unique closest 0-cube $\Gg_{K}(x) \in K$, called the \emph{gate} of $x$ in $K$. 
\end{definition}

The gate is characterized by the property that any hyperplane $H$ separates $\Gg_{K}(x)$ from $x$ if and only if $H$ separates $x$ from $K$. 

The convexity of $K$ allows us to extend the map $x \to \Gg_{K}(x)$ to a projection $\Gg_{K} \from X \to K$, which is a cubical map defined as follows.
Let $c$ be a $d$-dimensional cube of $X$ and let $H_{1}, H_{2} . . . , H_{d}$ be the collection of (pairwise-crossing) hyperplanes which cross $c$. 
Suppose that these are labeled so that $H_{1}, H_{2}...,H_{s}$ cross $K$, for some $0 \leq s \leq d$, and that $H_{s+1}, H_{s+2}...,H_{d}$ do not cross $K$. Then the 0-cubes of $c$ map by $\Gg_{K}$ to the 0-cubes of a uniquely determined $s$-dimensional cube $\Gg_{K}(c)$ of $K$ in which the hyperplanes $H_{1}, H_{2}...,H_{s}$ intersect, and there is a cubical collapsing map $c \simeq [-1, 1]^{d} \to [-1, 1]^{s} \simeq \Gg_{K}(c)$  extending the gate map on the 0-skeleton. 

\begin{definition}[Projection to the contact graph]  \label{D:contact-proj}
Let $K$ be a convex subcomplex of $X$.  Given a hyperplane $H$, let $\calN_\kappa(H)$ denote its \emph{carrier}, i.e., the union of all closed cubes intersecting $H$. 
For each $0$-cube $k \in K$, let $\{ H_i \}_{i \in \calI}$ be the collection of hyperplanes such that $k \in \calN_\kappa(H_i)$, and 
define $\rho_{K} \from K \to 2^{\mathcal{C}(K)}$ by setting $\rho_{K} (k) = \{ H_{i} \cap K\}_{{i \in \calI }}$. 
Let us now define the projection map $\pi_K : X \to 2^{\mathcal{C}(K)}$ by setting $\pi_{K} := \rho_{K} \circ \Gg_{K}$, where $\Gg_{K}(x)$ is the gate of $x$ in $K$.
\end{definition}

The following version of the bounded geodesic image theorem is inspired by  \cite[Proposition 4.2]{HHS1}.
Given a set $S \subseteq C(X)$, we use the notation $B_1(S)$ to denote the $1$-neighborhood of $S$.

\begin{lemma}[Bounded geodesic image theorem] \label{L:bounded-geodesic}
Let $X = X(\Gamma)$ be a universal Salvetti complex, let $J$ be a join and let $K \subseteq X(J)$  be a sub-Salvetti complex. Then if a path $\gamma$ in $X$
satisfies $\pi_X(\gamma) \cap B_1(\pi_X(J)) = \emptyset$, we have $\textup{diam } \pi_K(\gamma) \leq 1$.
\end{lemma}

\begin{proof}
Let $x, y$ be two points on $\gamma$. If $\pi_K(x) \neq \pi_K(y)$, then there exists a hyperplane $H$ in $K$ which separates $\mathfrak{g}_K(x)$ 
and $\mathfrak{g}_K(y)$. Let $H'$ be a hyperplane in $X$ such that $H = H' \cap K$. Then by convexity $H$ also separates $x$ and $y$, hence 
its projection to the contact graph $C(X)$ intersects the projection of $\gamma$. Since $H$ also intersects $J$, this contradicts 
the condition $\pi_X(\gamma) \cap B_1(\pi_X(J))  = \emptyset$.
\end{proof}

We also recall the notion of \emph{factor system} from  \cite{HHS1}.

\begin{definition} [\cite{HHS1}, Definition 8.1](Factor system). Let $X = X(\Gamma)$. 
A set of sub-complexes of $X$, denoted $\GF$, which satisfies the following is
called a \emph{factor system} in $X$:
\begin{enumerate}
\item $X \in \GF$.
\item Each $F \in  \GF$ is a nonempty convex sub-complex of $X$.
\item There exists $\delta \geq 1$ such that for all $x \in X^{(0)}$, at most $\delta$ elements of  $\GF$ contain $x$.
\item Every nontrivial convex sub-complex parallel to a combinatorial hyperplane of $X$ is
in $ \GF$.
\item There exists $\xi \geq 0$ such that for all $F, F' \in  \GF$, either $\Gg_{F}(F') \in  \GF$ or $\textup{diam}(\Gg_{F}(F') )\leq \xi$.
\end{enumerate}
\end{definition}

Associated with a factor $F \in \GF$ is a \emph{factored contact graph} $\hat{C}F$ defined as the contact graph of $F$ with each subgraph that is the contact graph of some smaller element
of $\GF$ coned off.

\begin{lemma}[\cite{HHS1}, Lemmas 2.6 and 8.19] \label{factorproperties}
Let $F, F'$ be two convex subcomplexes. Then:
\begin{enumerate}[i)]
\item \label{projectionisapoint}$\Gg_{F}(F')$ and $\Gg_{F'}(F)$ are parallel subcomplexes.
\item\label{nsim} If $F$ is not parallel to a subcomplex of $F'$, then 
\[ \textup{diam}_{\hat{C}F} (\pi_{F}(F')) \leq \xi+2.\]
\end{enumerate}
\end{lemma}

Let us remark that if $F$ and $F'$ are isometric, then ii) is true under the (seemingly weaker) assumption that $F$ is not parallel to $F'$.

\subsection*{Excursion geodesics}
It follows from Theorem \ref{T:converge} that almost every sample path $(w_{n})$ of a random walk on an irreducible right-angled Artin group converges to exactly one point $\xi$ in the visual boundary, and there is a unique \CAT geodesic ray $\gamma$ which connects the base-point with $\xi$.
In this case, we say that the sample path \emph{tracks} the geodesic ray $\gamma$. To build the connection between a sample path and the associated geodesic ray, we characterize geodesics by bounding their excursions.  

We say a geodesic ray $\gamma = \{ g_0, g_{1}, g_{2}, \dots , g_{n}, \dots \}$ with respect to the word metric   in $A(\Gamma)$ is a \emph{$\kappa$-excursion geodesic} if there exists a function $\kappa$ and a constant $C$ such that
 its projection to every maximal join $J$ subcomplex is bounded above by $C  \kappa(t)$. That is, we have: 
 \begin{equation}\label{excursion}
 \sup_{J} d_{s(J)} (\Gg_{s(J)}(g_{0}), \Gg_{s(J)}(g_{n})) \leq C  \kappa(\Vert g_n \Vert) 
  \end{equation}
where the supremum is taken over all maximal join subcomplexes $J \subseteq X(\Gamma)$.
The main result of this section is the following. 

\begin{proposition}\label{excursioniscontracting}
For any sublinear function $\kappa$, a $\kappa$--excursion geodesic is also a $\kappa$--contracting geodesic. 
\end{proposition}

In order to discuss the proof of this Proposition, let us recall that two hyperplanes $H_{1}, H_{2}$ are \emph{strongly separated} if there does not exist a hyperplane $H$ that intersects both $H_{1}$ and $H_{2}$. 
Given two hyperplanes $H_1$ and $H_2$, the \emph{bridge} $B$ between them is the union of all geodesic segments of minimal length between $H_1$ and $H_2$. 
We need the following properties about hyperplanes in the Salvetti complex:

\begin{lemma}[Properties of Strongly Separated Hyperplanes] \label{SSH}
 Let $u, v, w$ be vertices of $\Gamma$, and let $H_{u}, H_{v}, H_{w}$ be the associated hyperplanes that are dual to edges incident to the base-point of $X(\Gamma)$. Let $L_{v}$ denote the stabilizer of $H_{v}$, i.e. the group generated by the link $lk(v)$.
 \begin{enumerate}[1)]
\item Let $H_{1} = g_{1}H_{v}$ and $H_{2} = g_{2} H_{w}$. Then\label{stronglyseparated},
\begin{enumerate}
\item $H_{1}$ intersects $H_{2} \Leftrightarrow v, w$ commute and $g^{-1}_{1} g_{2} \in  L_{v} L_{w}$.
\item There exists $H_{3}$ intersecting both $H_{1}$ and $H_{2} \Leftrightarrow \exists  u \in st(v) \cap st(w)$ such that
$g^{-1}_{1} g_{2} \in  L_{v} L_{u}L_{w}$.
\end{enumerate}
\item Let $H_{1}, H_{2}$ be strongly separated hyperplanes in a universal Salvetti complex. \label{uniquegeodesic} The bridge $B$ between $H_1$ and $H_2$
consists of a single geodesic from $H_{1}$ to $H_{2}$. 

\item \label{closetobridge} There is a universal constant $C >1$, depending only on the dimension of $X(\Gamma)$, such that 
for any $x \in H_{1}$  and $ y \in H_{2}$,
\[ d(x, y) \geq \frac{1}{C} \left(d(x, B) + d(y, B) \right) - d(H_{1}, H_{2}) -4.\]
\end{enumerate}
\end{lemma}
\begin{proof}
1) and 2) are proven in (\cite{ssh}, Lemma 2.2 and Lemma 3.1).
3) is proven for word-metric geodesics in (\cite{ssh}, Lemma 2.3). However, for every \CAT geodesic, there exists a word-metric geodesic that lies in 
a $1$-neighbourhood of it and is a $(2,0)$-quasi-geodesic in the \CAT metric. Combined with the fact that a bridge is both a \CAT geodesic and a word-metric geodesic, 3) holds with a larger multiplicative constant.
\end{proof}

\begin{lemma}\label{SSHsequence}
Let $\gamma$ be a geodesic ray. Let $\{ S_{i} \}$ denote a maximal sequence of strongly separated hyperplanes crossed by $\gamma$. Then there exists a sequence of joins, denoted $\{J_{k}\}$,
 travelled by $\gamma$ such that for all $i$, if $S_{i} \in J_{k}$, then $S_{i+1} \in \bigcup_{l =1,2,3} J_{k+l}$.
\end{lemma}
\begin{proof}
Consider the sequence $(H_{1}, H_{2}, H_{3}, \dots)$ of hyperplanes crossed by $\gamma$. Let $H_{k}$ be the first hyperplane that is strongly separated from $H_{1} = H_{v}$.
 By Lemma~\ref{SSH}\eqref{stronglyseparated}, suppose $g_{k-1}H_{w} = H_{k-1}$, then $g_{k-1}$ lies in $L_{v} L_{u}L_{w}$ where $u \in  st(v) \cap st(w)$. Since $H_{w}$ is the next hyperplane,
  then $g_{k}$ lies in $L_{v} L_{u} L_{w} s_{w}$. By Remark~\ref{joincontainslink}, each link is contained in a join, thus there exists a sequence of joins travelled consecutively by $\gamma$
  such that if $H_{i} \in J_{i}$ then $H_{k} \in J_{i+2}$. Now repeat the process between $H_{k}$ and the first hyperplane that is strongly separated from $H_{k}$, say $H_{k'}$. It is possible that in this case the 
  three joins connecting $H_{k}$ and and $H_{k'}$ do not overlap with the joins that connect $H_{1}$ and $H_{k}$. In that case, consider $H_{k}$ to be the wall that is in both joins. Therefore from $H_{1}$ to $H_{k'}$ the ray $\gamma$ crosses $6$ joins, satisfying the claim that $S_{i+1} \in \bigcup_{l =1,2,3} J_{k+l}$.
\end{proof}

\begin{corollary}\label{finiteprojection}
Consider the sequence of joins produced in Lemma~\ref{SSHsequence} and denote it  $\{ J_{i}\}$. Then the projection of $J_{i}$ to $J_{i+5}$ is a point.
\end{corollary}
\begin{proof}
By Lemma~\ref{SSHsequence}, any geodesic connecting $J_{i}$ to its projection onto $J_{i+5}$ passes through at least 2 strongly separated hyperplanes. By Lemma~\ref{factorproperties}(\ref{projectionisapoint}), 
the projections of a pair of strongly separated hyperplanes to one another are parallel. But strong separability implies that both projections consist of a single point. 
Therefore, the projection of $J_{i}$ to $J_{i+5}$  is a point.
\end{proof}

\begin{lemma}[An excursion geodesic travels close to bridges] \label{L:bridge}
Fix a sublinear function $\kappa$, and let $\gamma$ be a $\kappa$-excursion geodesic ray with itinerary $\{ J_{i}\}$ as produced in Lemma~\ref{SSHsequence}.
Let $\{ S_{i}\}$ be the sequence of strongly separated hyperplanes in Corollary~\ref{SSHsequence}, and let $\calB_{i, j}$ be the bridge between $S_i$ and $S_{j}$. Let $b_{i}(j)$ denote the intersection point of $\calB_{i, j}$ with $S_{i}$. Also let $x_{i}$ be any point in the intersection $\gamma \cap S_{i}$.  Then, if $|i-j| =1$ we have
\[ d(x_{i}, b_{i}(j)) \leq C  \kappa(\Vert x_i \Vert). \] 
\end{lemma}

\begin{proof}
By Lemma~\ref{SSH}(\ref{uniquegeodesic}),  the bridge $\calB_{i, i+1}$ is a geodesic segment. By definition the length of a bridge is shorter than the distance between any other pair of points in $S_{i}$ and $S_{i+1}$. Since $\{ J_{i}\}$ is a $\kappa(t)$-itinerary, the lengths of bridges are bounded above by the lengths $d( x_{i}, x_{i+1})$, which is bounded by a constant multiple of $\kappa(t)$. Let that constant be $C$.
Since $\gamma$ is a $\kappa(t)$-excursion geodesic, $d(x_{i}, x_{i+1}) \leq \kappa(\Vert x_i \Vert)$. By Lemma~\ref{SSH}(\ref{closetobridge}), 
\begin{align*}
\frac{1}{C} (d(x_{i}, b_{i}(i+1)) + d(x_{i+1}, b_{i+1}(i)) - |\calB_{i,i+1}| -4 &\leq d( x_{i}, x_{i+1})   \leq    \kappa(\Vert x_i \Vert)  \\
\frac{1}{C} (d(x_{i}, b_{i}(i+1)) + d(x_{i+1}, b_{i+1}(i)) &\leq |\calB_{i,i+1}| + 4  +   \kappa( \Vert x_i \Vert)  \\
                                                                                    & \leq C  \kappa( \Vert x_i \Vert)                                                             
\end{align*}
Therefore $d(x_{i}, b_{i}(i+1))$ and $d(x_{i+1}, b_{i+1}(i)) $ are both bounded by $C  \kappa(\Vert x_i \Vert)$.
\end{proof}

Now we are ready to prove that the set of all $\kappa$-excursion geodesics is a subset of the $\kappa$--Morse boundary. 
We first replace an excursion geodesic with a geodesic in the \CAT metric that enters and leaves each maximal join at the same pair of points.

\begin{proof}[Proof of Proposition \ref{excursioniscontracting}]
Let $\gamma$ be a $\kappa(t)$--excursion geodesic and let $\{ J_{i}\}$ be the associated itinerary of joins produced in Lemma \ref{SSHsequence}. Let $x$ be in a maximal join $J_{i}$ with $x \notin \gamma$, let 
\[A : = \bigcup_{k=0}^{5} J_{i+k},\]
 and $\overline{A} : = A \cup J_{i-1} \cup J_{i+6}$.  Consider now a metric ball $\Sigma := \{ y \in X(\Gamma) \ : \ d(x,  y) < d(x, \gamma) \}$ which is disjoint from $\gamma$. Our goal is to prove 
 that for any $y \in \Sigma$ we have $d(x_\gamma, y_\gamma) \leq C \kappa( \Vert x \Vert )$, where $x_\gamma$ denotes the closest-point projection of the point $x$ to $\gamma$.
 
Let $y \in \Sigma$. If $y \in A$, then there exists a constant $C_{1}$ such that
\[ d(x_{\gamma}, y_{\gamma}) \leq C_{1} \kappa( \Vert x \Vert). \]
Otherwise, consider $y$ in $J_{i+k}$, $k \geq 6$.  There exists points $p \in J_{i + k}$ and closest to $x$ such that
\begin{equation}\label{gatepoint}
d(x, y) = d(x, p) + d(p, y).
\end{equation}
 
That is to say $p \in \Gg_{J_{i+6}} (J_{i})$. By Corollary~\ref{finiteprojection}, $p$ is unique and therefore $p = \calB_{i+6} ({i+5})$, thus by Lemma~\ref{L:bridge}, there exists constant $C_{2}$ such that
\[
d(p, \gamma) \leq C_{2 } \kappa(\Norm x ).
\]
By way of contradiction, suppose $d(p, y) \geq d(p, \gamma)$. Then
\begin{align*}
d(x, y) &= d(x, p) + d(p, y) \qquad \text{by eq.~\eqref{gatepoint}} \\
           &\geq d(x, p) + d(p, \gamma)\\
           &\geq d(x, \gamma).
\end{align*}
This is contrary to our assumption that $y \in \Sigma$. Therefore, $d(p, y) < d(p, \gamma)$, hence $d(p, y) < C_{2} \kappa( \Norm x)$.
 By the Lipschitz property of \CAT projections, we have
\[
d(p_{\gamma}, y_{\gamma}) \leq C_{2} \kappa(\Norm x).
\]
Since 
$d(x_{\gamma}, p_{\gamma}) \leq C_{1} \kappa(\Norm x)$, then 
 \[d(x_{\gamma}, y_{\gamma})  \leq  d(x_{\gamma}, p_{\gamma})  +  d (p_{\gamma}, y_{\gamma})  \leq (C_{1}+C_{2}) \kappa(\Norm x) .\]
\end{proof}

\subsection*{Excursion of random geodesics}

To show that the $\kappa$--Morse boundary has full measure, we need to control the excursion of the random walk in each sub-join of the Salvetti complex. 
We will use the following variation of the main theorem in  \cite{relativehyperbolic} (we thank Sam Taylor for suggesting the argument).

\begin{theorem}  \label{T:log-exc}
Let $\mu$ be a finitely supported, generating probability measure on an irreducible right-angled Artin group $A(\Gamma)$. 
Then for any $k > 0$ there exists $C > 0$ such that for all $n$ we have
$$\mathbb{P}\left(\sup_J d_{s(J)}(1, w_n) \geq C \log n \right) \leq C n^{-k},$$
where the supremum is taken over all join subcomplexes of $X(\Gamma)$.

As a consequence, for almost every sample path there exists $C > 0$ such that for all $n$
$$\sup_J d_{s(J)}(1, w_n) \leq C \log n.$$
\end{theorem}

\begin{proof}
The idea of the proof is that in order to make progress in $s(J)$, the sample path must project close to the projection of $J$ in the contact graph $C(X)$. 
However, linear progress with exponential decay implies that the sample path can stay close to the projection of $J$ only for a time of order $\log n$, 
which completes the proof. 

Let us see the details. By linear progress with exponential decay  \cite[Theorem 1.2]{Maher-Exp}, there exists $L > 0$ and $C_1$ such that
$$\mathbb{P}(d_{C(X)}(1, w_n) \leq L n) \leq C e^{-n/C}$$
for all $n$, so for any $A > 0$ and $n \geq e^{2/LA}$ we have 
\begin{equation} \label{E:exp-decay}
\mathbb{P}(d_{C(X)}(1, w_{A \log n}) \leq 2) \leq C n^{-A/C}.
\end{equation}

By the distance formula  \cite[Theorem 9.1]{HHS1}, for any $B > 0$ there exist $C_2$ such that 
for any join $J$
$$d_{s(J)}(x, y) \leq C_2 \sum_{K \subseteq J} \{ d_{C(K)}(x, y) \}_B + C_2$$
where $\{ x \}_B = x$ if $x \geq B$, and $\{ x \}_B =  0$ otherwise. In particular, by Lemma \ref{L:bounded-geodesic} there exists $C_2$ such that if a path $\gamma = [x, y]$ projects far from $J$ in $C(X)$ then 
$$d_{s(J)}(x, y) \leq C_2.$$

Consider now the path of vertices $(w_i)_{i \leq n}$ in $X(\Gamma)$, and suppose that for a join $J$ we have $d_{s(J)}(1, w_n) \geq C \log n$.
Let 
\[
i_1 := \min \{ 0 \leq i \leq n \ : \ d_{C(X)}(w_i, J) \leq 1 \},
\] 
\[
i_2 := \max \{ 0 \leq i \leq n \ : \ d_{C(X)}(w_i, J) \leq 1 \},
\] 
and 
\[
D := \max\{d_{X(\Gamma)}(1, g) \ : \ g \in \supp \ \mu \}.
\] 
Then 
$$C \log n  \leq d_{s(J)}(1, w_n) \leq d_{s(J)}(1, w_{i_1}) + d_{s(J)}(w_{i_1}, w_{i_2}) + d_{s(J)}(w_{i_2}, w_n) \leq D (i_2 - i_1) + 2 C_2$$
hence, for $n$ large enough, 
$$|i_1 - i_2| \geq \frac{C \log n}{2 D}.$$
Hence
$$\mathbb{P}(\exists J : \ d_{s(J)}(1, w_n) \geq C \log n) \leq \mathbb{P}\left(\exists i_1 \leq i_2 \leq n,  i_2 - i_1 \geq \frac{C}{2D} \log n \ : \ d_{C(X)}(w_{i_1}, w_{i_2}) \leq 2 \right)$$
and by \eqref{E:exp-decay} this is bounded above by 
$$\leq n^2 \cdot C_1 n^{- \frac{C}{2D C_1}}$$
which tends to $0$ for $n \to \infty$ as long as $C > 4DC_1$. 

The second claim follows immediately from the first one for $k = 2$ by Borel-Cantelli.
\end{proof}

\subsection*{Sublinear deviation between geodesics and sample paths}
What remains is to understand the Hausdorff distance between a sample path and the geodesic ray that it is tracking. 
We first recall the following result:

\begin{theorem}[\cite{sisto}, Theorem 5.2] \label{T:mcg-track}
Let $S$ be a connected, orientable surface of finite type,
with empty boundary and complexity at least 2. Let $M(S)$ be its mapping
class group and let $\{w_{n}\}$ be a random walk on $M(S)$ driven by a finitely supported measure $\mu$. Then for any $k \geq 1$ 
there exists a constant $C$ such that 
$$\mathbb{P}\left( \sup d_{\rm Haus}(\{w_{i}\}_{i \leq n}, \gamma(w_{n}) ) 
  \geq C\sqrt{n \log n} \right) \leq C n^{-k}$$
where the supremum is taken over all geodesics in a given word metric and hierarchy paths $\gamma(w_{n})$ from 1 to $w_{n}$.
\end{theorem}

It turns out that the proof in \cite{sisto} uses all ingredients that are known for right-angled Artin groups, namely 
the bounded geodesic image theorem, the distance formula, and quadratic divergence (where the $\sqrt{n}$ function comes from). 
Hence the same proof as in \cite{sisto} yields: 

\begin{theorem} \label{T:raag-track}
Let $G$ be an irreducible right-angled Artin group and let $\{w_{n}\}$ be a random walk on $G$ driven by a finitely supported, generating measure $\mu$. 
Then for any $k \geq 1$ there exists a constant $C$ such that 
$$\mathbb{P}\left( \sup d_{\rm Haus}\big(\{w_{i}\}_{i \leq n}, \gamma(w_{n}) \big) 
  \geq C\sqrt{n \log n} \right) \leq C n^{-k}$$
where the supremum is taken over all geodesics in a given word metric from 1 to $w_{n}$.
\end{theorem}

As a consequence, we obtain the following tracking estimate. 

\begin{proposition} \label{P:deviation}
Let $\mu$ be a finitely supported, generating measure on an irreducible right-angled Artin group $G = A(\Gamma)$, with universal Salvetti complex $X = X(\Gamma)$. 
Then there exists $\ell > 0$ such that for almost every sample path $(w_n)$ there exists a \CAT geodesic ray $\gamma$ in $X$ starting 
at the base-point such that 
$$\limsup_{n \to \infty} \frac{d_X(w_n, \gamma(\ell n))}{\sqrt{n \log n} } < +\infty.$$
\end{proposition}

\begin{proof}
Since $G = A(\Gamma)$ is non-amenable and its action on $X$ is cocompact, there exists $\ell > 0$ such that for almost every sample path
\begin{equation} \label{E:progress}
\lim_{n \to \infty} \frac{d_X(1, w_n)}{n} = \ell.
\end{equation}
Now, by Theorem \ref{T:raag-track}, there exists $C > 0$ such that for any $n$
$$\mathbb{P}\left( \sup_{i \leq n} d_X( w_i, \gamma_n) \geq C\sqrt{n \log n} \right) \leq C n^{-2}$$
where $\gamma_n$ is the \CAT geodesic joining $1$ and $w_n$. Hence, by Borel-Cantelli for almost every sample path 
there exists a constant $C'$ such that 
\begin{equation} \label{E:sqrtn}
\sup_{i \leq n} d_X( w_i, \gamma_n) \leq C' \sqrt{n \log n}
\end{equation}
for any $n$. 
Now, let $n_k := e^k n$ and consider the triangle with vertices $\{ w_{n_{k-1}}, 1, w_{n_k} \}$. 
For large $n$ and any $k \geq 1$, we have $d_X(w_{n_{k-1}}, \gamma_{n_k}) \leq C' \sqrt{n_k \log n_k }$ by \eqref{E:sqrtn} and $d_X(1, w_{n_{k-1}}) \geq \frac{\ell}{2} n_{k-1}$ by \eqref{E:progress}, then 
by comparison with a euclidean triangle,
$$d_X(\gamma_{n_{k-1}}(\ell n), \gamma_{n_{k}}(\ell n)) \lesssim \ell n \frac{d_X(w_{n_{k-1}}, \gamma_{n_{k}})}{d_X(1, w_{n_{k-1}})}  \lesssim \sqrt{n \log n} \sqrt{k} e^{-k/2}.$$
Hence 
$$d_X(w_n, \gamma(\ell n)) \lesssim \sum_{k=1}^\infty d_X(\gamma_{n_{k-1}}(\ell n), \gamma_{n_{k}}(\ell n)) \lesssim \sqrt{n \log n}$$
which proves the claim.
\end{proof}

\subsection*{Proof of Theorem~\ref{T:full-measure} and Theorem~\ref{T:intro-Poisson}}

Recall that almost every sample path $(w_n)$ converges to a point $\xi$ in the visual boundary: let $\gamma$ be the infinite \CAT geodesic 
connecting the base-point to $\xi$, and let $\gamma' = \{g_0, g_1, \dots \}$ be a combinatorial geodesic in $X(\Gamma)$ which lies at distance at most $1$ from $\gamma$.

By Theorem~\ref{T:log-exc}, for almost every sample path there exists $C_1 > 0$ such that
$$d_{s(J)}(1, w_n) \leq C_1 \log n.$$
for any join $J$. Moreover, by Proposition \ref{P:deviation}, for almost every sample path and any $J$,
$$d_{s(J)}(\gamma(\ell n), w_n) \leq d_{X(\Gamma)}(\gamma(\ell n), w_n) \leq C_2 \sqrt{n \log n}$$
hence, since $g_n$ lies within distance $1$ of $\gamma(n)$, 
$$d_{s(J)}(1, g_n) \leq C_1 \log (n/\ell) + C_2 \sqrt{(n/\ell) \log (n/\ell)} \leq C_3 \sqrt{n \log n}.$$
Thus, the geodesic $\gamma'$ is a $\kappa$-excursion geodesic with $\kappa = \sqrt{t \log t}$, hence it is also a $\kappa$-contracting geodesic. 
This proves Theorem~\ref{T:full-measure}, hence also Theorem~\ref{T:intro-Poisson}.

\bibliographystyle{alpha}

\end{document}